\documentclass[english,11pt]{smfart}

\usepackage[utf8]{inputenc}
\usepackage[british]{babel}

\usepackage[mono=false]{libertine}
\usepackage[T1]{fontenc}

\usepackage{amsthm}
\usepackage[cmintegrals,libertine]{newtxmath}
\usepackage{calrsfs}
\usepackage[cal = pxtx, calscaled=.9, scr=boondoxo]{mathalfa}
\useosf

\usepackage{microtype}

\usepackage{numprint}

\usepackage[a4paper,vmargin={3cm,3cm},hmargin={2.5cm,2.5cm}]{geometry}
\usepackage{tikz}
	\usetikzlibrary{calc}
	\usetikzlibrary{decorations.markings}

\usepackage{graphicx}

\usepackage[font=sf]{caption}

\usepackage{hyperref}
\hypersetup{
	colorlinks=true,
		citecolor=blue!60!black,
		linkcolor=red!60!black,
		urlcolor=green!40!black,
		filecolor=yellow!50!black,
	breaklinks=true,
	pdfpagemode=UseNone,
	bookmarksopen=false,
}

\usepackage{enumitem}
	\renewcommand{\theenumi}{{(\roman{enumi})}}
	\renewcommand{\labelenumi}{\theenumi}

\newtheorem{thm}{Theorem}
\newtheorem{lem}{Lemma}
\newtheorem{prop}{Proposition}

\theoremstyle{definition}
\newtheorem{rem}{Remark}

\newcommand{\ensembles}[1]{\mathbf{#1}}
	\newcommand{\N}{\ensembles{N}}
	\newcommand{\Z}{\ensembles{Z}}
	\newcommand{\R}{\ensembles{R}}
	
	\renewcommand{\P}{\ensembles{P}}
	\newcommand{\E}{\ensembles{E}}
	\newcommand{\Var}{\mathrm{Var}}

\DeclareMathAlphabet{\mathbbo}{U}{bbold}{m}{n}
	\newcommand{\1}{\mathbbo{1}}
\newcommand{\ind}[1]{\1_{\{#1\}}}

\renewcommand{\Pr}[1]{\P\left(#1\right)}
\newcommand{\Prc}[2]{\P\left(#1 \;\middle|\; #2\right)}
\newcommand{\Es}[1]{\E\left[#1\right]}
\newcommand{\Esc}[2]{\E\left[#1 \;\middle|\; #2\right]}

\newcommand{\Map}{\mathbf{M}}
\newcommand{\Mapp}{\Map^\bullet}

\newcommand{\CRT}{\mathscr{T}}

\newcommand{\Bmap}{\mathscr{M}}
\newcommand{\dBmap}{\mathscr{D}}
\newcommand{\pBmap}{\mathscr{m}}

\newcommand{\dgr}{d_{\mathrm{gr}}}
\newcommand{\pgr}{p_{\mathrm{gr}}}

\renewcommand{\d}{\mathrm{d}}
\newcommand{\e}{\mathrm{e}}
\newcommand{\Xexc}{\mathscr{X}}
\newcommand{\Hexc}{\mathscr{H}}
\newcommand{\Lab}{\mathscr{L}}
\newcommand{\Snake}{\mathscr{S}}

\newcommand{\q}{\mathbf{q}}

\renewcommand{\i}{\mathrm{i}}

\newcommand{\cv}[1][n]{\enskip\mathop{\longrightarrow}^{}_{#1 \to \infty}\enskip}
\newcommand{\cvloi}[1][n]{\enskip\mathop{\longrightarrow}^{(d)}_{#1 \to \infty}\enskip}

\newcommand{\cvproba}[1][n]{\enskip\mathop{\longrightarrow}^{\P}_{#1 \to \infty}\enskip}

\newcommand{\eqloi}{\enskip\mathop{=}^{(d)}\enskip}

\title{On scaling limits of planar maps with stable face-degrees}

\author{Cyril \textsc{Marzouk}}
\address{Laboratoire de Mathématiques d’Orsay, Univ. Paris-Sud, CNRS, Université Paris-Saclay, 91405 Orsay, France.}
\email{cyril.marzouk@u-psud.fr}

\begin{document}

\maketitle

\begin{abstract}
We discuss the asymptotic behaviour of random critical Boltzmann planar maps in which the degree of a typical face belongs to the domain of attraction of a stable law with index $\alpha \in (1,2]$. We prove that when conditioning such maps to have $n$ vertices, or $n$ edges, or $n$ faces, the vertex-set endowed with the graph distance suitably rescaled  converges in distribution towards the celebrated Brownian map when $\alpha=2$, and, after extraction of a subsequence, towards another `$\alpha$-stable map'  when $\alpha <2$, which improves on a first result due to Le Gall \& Miermont who assumed slightly more regularity.
\end{abstract}

\begin{figure}[!ht] \centering
\includegraphics[width=.45\linewidth, trim=6cm 7cm 10cm 6cm, clip]{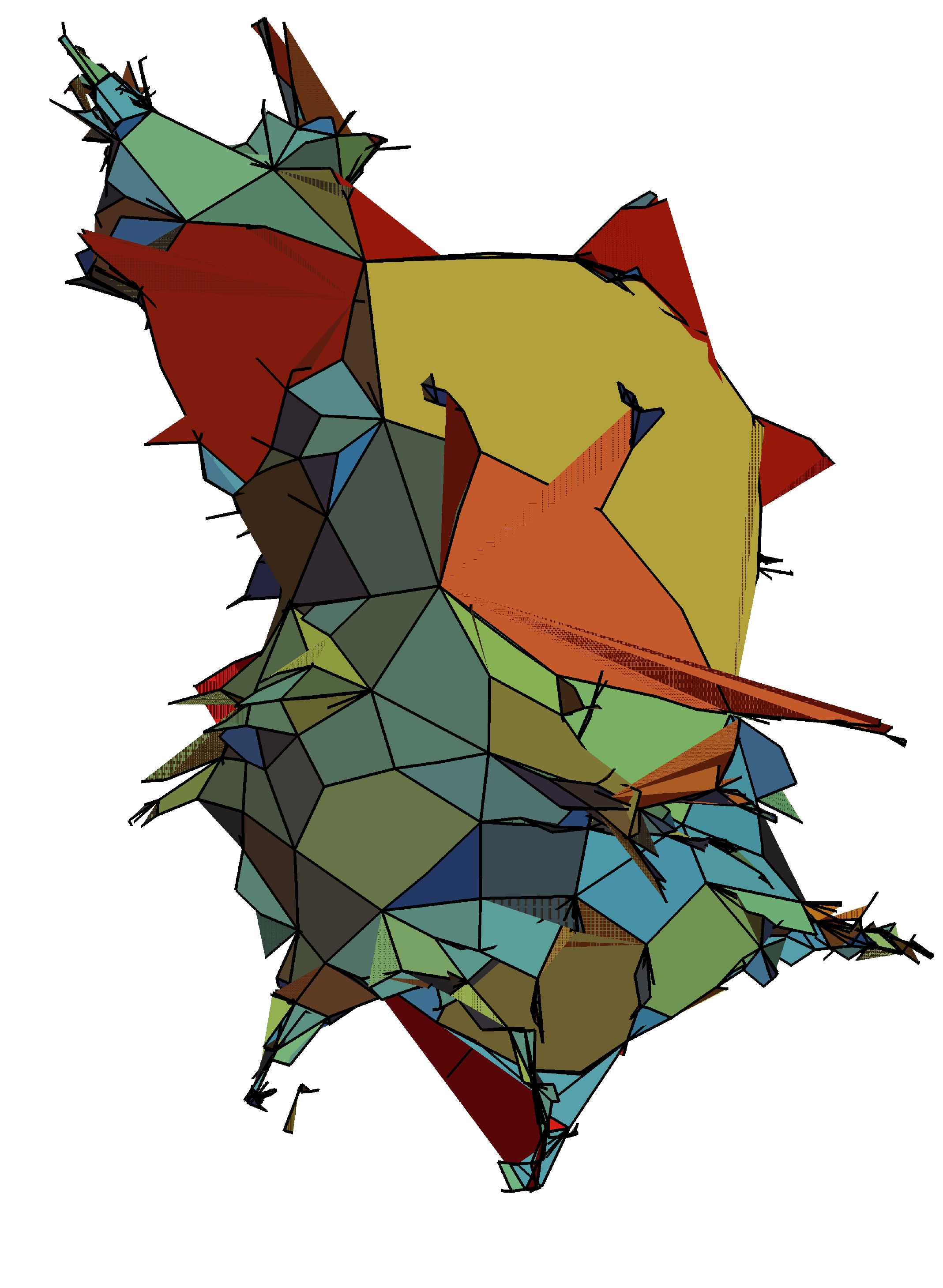} 
\quad
\includegraphics[width=.45\linewidth, trim=7cm 17cm 6cm 8cm, clip]{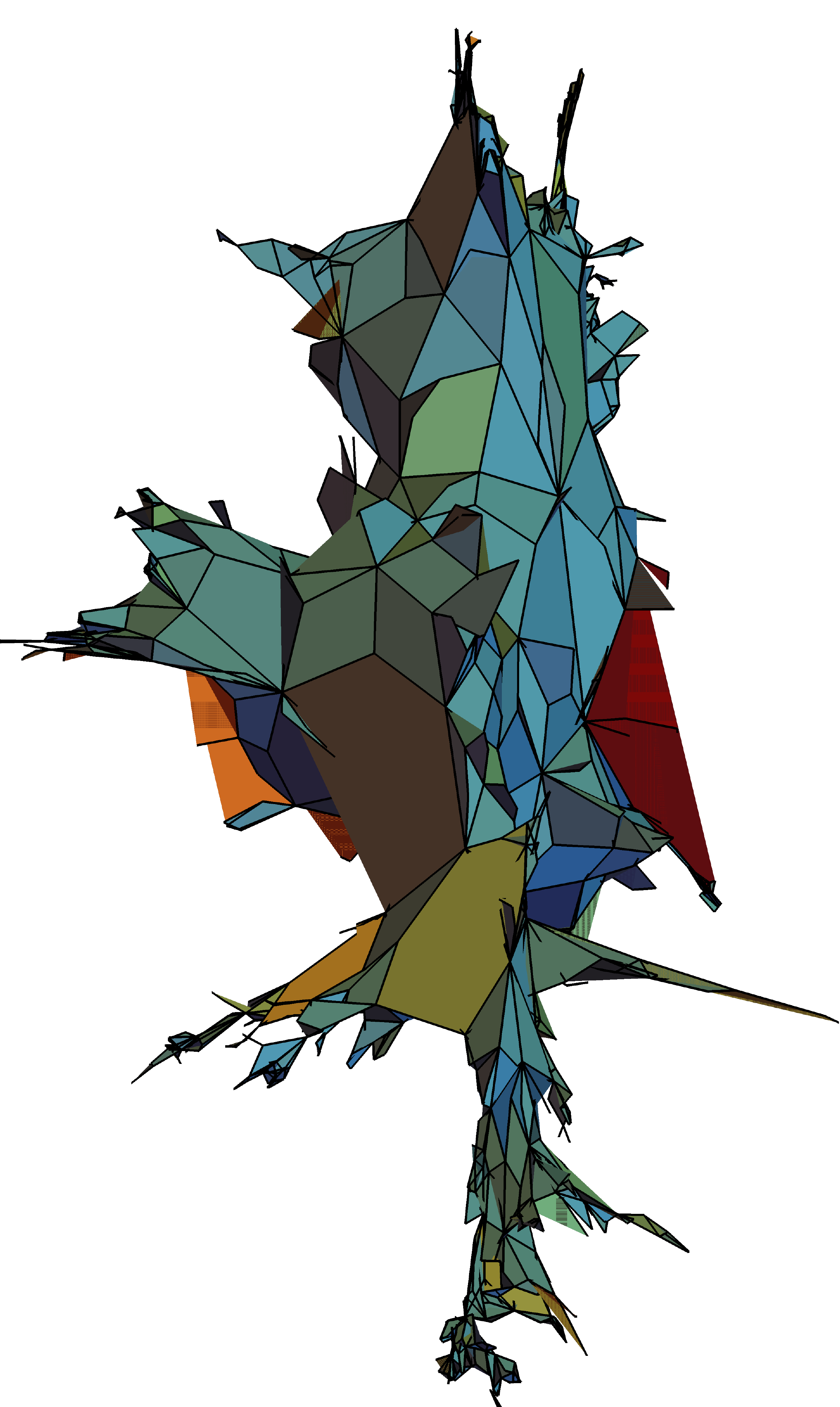}
\caption{Simulations of large $\alpha$-stable Boltzmann maps with $\alpha=\numprint{1.7}$ on the left and $\alpha=\numprint{1.9}$ on the right. Courtesy of Nicolas Curien.}
\label{fig:simu_cartes_stables}
\end{figure}

\section{Introduction and main result}

This work deals with scaling limits of large random planar maps viewed as metric measured spaces. We assume that the reader is already acquainted with this theory; let us describe the precise model that we consider before stating our main results.

We study rooted planar maps, which are finite (multi-)graphs embedded in the two-dimensional sphere, viewed up to homeomorphisms, and equipped with a distinguished oriented edge called the \emph{root-edge}. For technical reasons, we restrict ourselves to \emph{bipartite} maps which are those maps in which all faces have even degree. Given a sequence $\q= (q_{k})_{ k \geq 1}$ of non-negative numbers such that $q_k \ne 0$ for at least one $k \ge 3$ (in order to discard trivial cases), we define a Boltzmann measure $w^\q$ on the set $\Map$ of all finite bipartite maps by assigning a weight:
\[w^\q(M) \enskip=\enskip \prod_{f\in \mathsf{Faces}(M)} q_{ \deg(f)/2},\]
to each such map $M$. We shall also consider rooted and \emph{pointed} maps in which we distinguish a vertex $\star$ in a map $M$; we then define a pointed Boltzmann measure on the set $\Map^\bullet$ of pointed maps by setting $w^{\q, \bullet}(M, \star) = w^\q(M)$. Let $W^\q = w^\q(\Map)$ and $W^{\q, \bullet} = w^{\q, \bullet}(\Map^\bullet)$ be their total mass; obviously the latter is greater than the former, but Bernardi, Curien \& Miermont~\cite{Bernardi-Curien-Miermont:A_Boltzmann_approach_to_percolation_on_random_triangulations} proved that
\[W^\q < \infty \quad\text{if and only if}\quad W^{\q, \bullet} < \infty.\]
When this holds, we say that the sequence $\q$ is \emph{admissible} and we normalise our measures into probability measures $\P^\q$ and $\P^{\q, \bullet}$ respectively. We assume further that $\q$ is \emph{critical}, which means that the number of vertices of a map has infinite variance under $\P^\q$, or equivalently infinite mean under $\P^{\q, \bullet}$.

Such models of random maps have been first considered by Marckert \& Miermont~\cite{Marckert-Miermont:Invariance_principles_for_random_bipartite_planar_maps} who gave analytic admissibility and criticality criteria, recast in~\cite{Marzouk:Scaling_limits_of_random_bipartite_planar_maps_with_a_prescribed_degree_sequence}, and which we shall recall later. 
Following the terminology introduced in the very recent work of Curien \& Richier~\cite{Curien-Richier:Duality_of_random_planar_maps_via_percolation}, we further assume that there exists $\alpha \in (1,2]$ such that our distributions are \emph{discrete stable with index $\alpha$}, which we define as follows: Under the pointed law $\P^{\q, \bullet}$, the degree of the face adjacent to the right of the root-edge (called the \emph{root-face}) belongs to the domain of attraction of a stable law with index $\alpha$. It can be checked that this degree under the non-pointed law $\P^\q$ is more regular, and under this assumption has finite variance for every $\alpha \in (1,2]$.
We shall interpret the law of this degree as that of a typical face in a large pointed or non-pointed Boltzmann random map.
Such an assumption was first formalised by Richier~\cite{Richier:Limits_of_the_boundary_of_random_planar_maps} (except that the case $\alpha=2$ was restricted to finite variance) and is more general than the one used e.g. in~\cite{Le_Gall-Miermont:Scaling_limits_of_random_planar_maps_with_large_faces,Marzouk:Scaling_limits_of_random_bipartite_planar_maps_with_a_prescribed_degree_sequence}.

For every integer $n \ge 2$, let $\Map_{E=n}$, $\Map_{V=n}$ and $\Map_{F=n}$ be the subsets of $\Map$ of those maps with respectively $n-1$ edges, $n+1$ vertices (these shifts by one will simplify the statements) and $n$ faces. For every $S = \{E, V, F\}$ and every $n \ge 2$, we define
\[\P^\q_{S=n}(M) = \P^\q(M \mid M \in \Map_{S=n}),
\qquad M \in \Map_{S=n},\]
the law of a rooted Boltzmann map conditioned to have `size' $n$. We define similarly pointed laws $\P^{\q, \bullet}_{S=n}$. Let us denote by $\zeta(M_n)$ the number of edges of the map $M_n$ sampled from such a law; note that it equals $n-1$ if $S=E$ but it is random otherwise.
We shall implicitly assume that the support of $\q$ generates the whole group $\Z$, not just a strict subgroup, so these laws are well-defined for every $n$ large enough; the general case only requires mild modifications.
We consider limits of large random maps in the following sense: given a finite map $M$, we endow its vertex-set (which we still denote by $M$) with the graph distance $\dgr$ and the uniform probability measure $\pgr$; the topology we use is then that given by the so-called \emph{Gromov--Hausdorff--Prokhorov} distance which makes the space of compact metric measured spaces (viewed up to isometries) a Polish space, see e.g. Miermont~\cite{Miermont:Tessellations_of_random_maps_of_arbitrary_genus}.

\begin{thm}
\label{thm:cv_cartes}
There exists an increasing sequence $(B_n)_{n \ge 1}$ such that the following holds. Fix $S \in \{E, V, F\}$ and for every $n \ge 2$, sample $M_n$ from $\P^\q_{S=n}$ or from $\P^{\q, \bullet}_{S=n}$, then:
\begin{enumerate}
\item If $\alpha=2$, then we have the convergence in distribution in the sense of Gromov--Hausdorff--Prokhorov
\[\left(M_n, B_{\zeta(M_n)}^{-1/2} \dgr, \pgr\right) \cvloi (\Bmap, \dBmap, \pBmap),\]
where $(\Bmap, (\frac{9}{8})^{1/4} \dBmap, \pBmap)$ is the (standard) \emph{Brownian map}.

\item If $\alpha<2$, then from every increasing sequence of integers, one can extract a subsequence along which we have the convergence in distribution in the sense of Gromov--Hausdorff--Prokhorov,
\[\left(M_n, B_{\zeta(M_n)}^{-1/2} \dgr, \pgr\right) \cvloi (\Bmap, \dBmap, \pBmap),\]
where $(\Bmap, \dBmap, \pBmap)$ is a random compact measured metric space with Hausdorff dimension $2\alpha$.
\end{enumerate}
\end{thm}

\begin{rem}\label{rem:nombre_aretes_carte}
\begin{enumerate}
\item This result is reminiscent of the work of Duquesne~\cite{Duquesne:A_limit_theorem_for_the_contour_process_of_conditioned_Galton_Watson_trees} and Kortchemski~\cite{Kortchemski:Invariance_principles_for_Galton_Watson_trees_conditioned_on_the_number_of_leaves} on size-conditioned Bienaymé--Galton--Watson trees (see~\eqref{eq:Duquesne_Kortchemski} below) and indeed, the sequence $(B_n)_{n \ge 1}$ is the same as there; it is of order $n^{1/\alpha}$, and in the finite-variance regime, it takes the form $B_n = (n \sigma^2/2)^{1/2}$ for some $\sigma^2 \in (0,\infty)$.

\item We shall see in Remark~\ref{rem:nombre_aretes_arbre} that under $\P^\q_{S=n}$ or $\P^{\q, \bullet}_{S=n}$ we have for some constant $Z_\q > 1$
\[n^{-1} \zeta(M_n) \cvproba Z_\q \enskip\text{if}\enskip S=V
\qquad\text{and}\qquad
n^{-1} \zeta(M_n) \cvproba (1-Z_\q^{-1})^{-1} \enskip\text{if}\enskip S=F,\]
so the factor $B_{\zeta(M_n)}^{-1/2}$ may be replaced by $Z_\q^{-1/(2\alpha)} B_n^{-1/2}$ and $(1-Z_\q^{-1})^{1/(2\alpha)} B_n^{-1/2}$ respectively.
\end{enumerate}
\end{rem}

In the Gaussian case $\alpha=2$, tightness in the sense of Gromov--Hausdorff of rescaled uniform random \emph{$2\kappa$-angulations} (all faces have degree $2\kappa$ fixed) with $n$ faces was first obtained by Le Gall~\cite{Le_Gall:The_topological_structure_of_scaling_limits_of_large_planar_maps}. The Brownian map was then characterised independently by Le Gall~\cite{Le_Gall:Uniqueness_and_universality_of_the_Brownian_map} and Miermont~\cite{Miermont:The_Brownian_map_is_the_scaling_limit_of_uniform_random_plane_quadrangulations} which yields the convergence of these maps; building upon the pioneer work of Marckert \& Miermont~\cite{Marckert-Miermont:Invariance_principles_for_random_bipartite_planar_maps}, Le Gall~\cite{Le_Gall:Uniqueness_and_universality_of_the_Brownian_map} also includes Boltzmann planar maps conditioned by the number of vertices, assuming exponential moments. This assumption was then reduced to a second moment in~\cite{Marzouk:Scaling_limits_of_random_bipartite_planar_maps_with_a_prescribed_degree_sequence}, as a corollary of a more general model of random maps `with a prescribed degree sequence'. Let $\dBmap^\ast = (\frac{9}{8})^{1/4} \dBmap$, then in this finite variance regime, Theorem~\ref{thm:cv_cartes} reads thanks to the preceding remark:
\[\left(M_n, \left(\frac{9}{4 \sigma^2 \zeta(M_n)}\right)^{1/4} \dgr, \pgr\right) \cvloi (\Bmap, \dBmap^\ast, \pBmap),\]
which recovers~\cite[Theorem~3]{Marzouk:Scaling_limits_of_random_bipartite_planar_maps_with_a_prescribed_degree_sequence}.

In the case $\alpha < 2$, Theorem~\ref{thm:cv_cartes} extends a result due to Le Gall and Miermont~\cite{Le_Gall-Miermont:Scaling_limits_of_random_planar_maps_with_large_faces} who studied the Gromov--Hausdorff convergence of such maps conditioned by the number of vertices in the particular case where the probability that the root-face has degree $2k$ under $\P^{\q, \bullet}$ equals $C k^{-\alpha-1} (1+o(1))$ for some constant $C > 0$. Because the conjectured `stable maps' have not yet been characterised, the extraction of a subsequence is needed in Theorem~\ref{thm:cv_cartes}. Nonetheless, as in~\cite{Le_Gall-Miermont:Scaling_limits_of_random_planar_maps_with_large_faces}, we derive some scaling limits which do not necessitate such an extraction: in Theorem~\ref{thm:profil} below, we give the limit of the maximal distance to the distinguished vertex in a pointed map, or to a uniformly chosen vertex in the non-pointed map, as well as the \emph{profile} of the map, given by the number of vertices at distance $k$ to such a vertex, for every $k \ge 0$. Let us finally mention the work of Richier~\cite{Richier:Limits_of_the_boundary_of_random_planar_maps} and more recently with Kortchemski~\cite{Kortchemski-Richier:The_boundary_of_random_planar_maps_via_looptrees} who analyse the geometric behaviour of the \emph{boundary} of the root-face when conditioned to be large, and so, roughly speaking, the geometric behaviour of macroscopic faces of the map.

\begin{rem}
\begin{enumerate}
\item As in~\cite{Marzouk:Scaling_limits_of_random_bipartite_planar_maps_with_a_prescribed_degree_sequence}, the proof of Theorem~\ref{thm:cv_cartes} actually shows that we can also take as notion of size of a map the number of faces whose degree belongs to a fixed subset $A \subset 2\N$, at least when either $A$ or its complement is finite.
\item As observed in~\cite[Theorem~4]{Marzouk:Scaling_limits_of_random_bipartite_planar_maps_with_a_prescribed_degree_sequence}, the conditioning by the number of edges is special since $\Map_{E=n}$ is a finite set for every $n$ fixed so we may define the law $\P^\q_{E=n}$ even when $\q$ is not admissible and our results still holds under appropriate assumptions.
\item As in~\cite{Le_Gall-Miermont:Scaling_limits_of_random_planar_maps_with_large_faces}, Theorem~\ref{thm:cv_cartes} and the other main results below hold when conditioning the maps to have `size' \emph{at least} $n$, the references cover this case and the proofs only require mild modifications.
\end{enumerate}
\end{rem}

The proof of convergences as in Theorem~\ref{thm:cv_cartes} in~\cite{Le_Gall:Uniqueness_and_universality_of_the_Brownian_map, Le_Gall-Miermont:Scaling_limits_of_random_planar_maps_with_large_faces} relied on a bijection due to Bouttier, Di Francesco \& Guitter~\cite{Bouttier-Di_Francesco-Guitter:Planar_maps_as_labeled_mobiles} which shows that a pointed map is encoded by `two-type' \emph{labelled tree} and one of the key steps was to prove that this labelled tree, suitably rescaled, converges in distribution towards a `continuous' limit which similarly describes the limit $(\Bmap, \dBmap, \pBmap)$. In~\cite{Marzouk:Scaling_limits_of_random_bipartite_planar_maps_with_a_prescribed_degree_sequence} we studied this two-type tree by further relying on a more recent work of Janson--Stef{\'a}nsson~\cite{Janson-Stefansson:Scaling_limits_of_random_planar_maps_with_a_unique_large_face} who established a bijection between these a `two-type' trees and `one-type' trees which are much easier to control. The scheme of the proof of the analogous statement in~\cite{Marzouk:Scaling_limits_of_random_bipartite_planar_maps_with_a_prescribed_degree_sequence} was first to prove that this `one-type' labelled tree converges towards a continuous object, then transporting this convergence to the two-type tree and finally conclude from the arguments developed in~\cite{Le_Gall:Uniqueness_and_universality_of_the_Brownian_map, Le_Gall-Miermont:Scaling_limits_of_random_planar_maps_with_large_faces}.

In this paper, we bypass the bijection~\cite{Bouttier-Di_Francesco-Guitter:Planar_maps_as_labeled_mobiles} and only work with the one-type tree from~\cite{Janson-Stefansson:Scaling_limits_of_random_planar_maps_with_a_unique_large_face}; we prove the convergence of this object in Theorem~\ref{thm:cv_serpents_cartes} and deduce Theorem~\ref{thm:cv_cartes} by recasting the arguments from~\cite{Le_Gall:Uniqueness_and_universality_of_the_Brownian_map, Le_Gall-Miermont:Scaling_limits_of_random_planar_maps_with_large_faces}. On the one-hand, the advantage of the bijection from~\cite{Bouttier-Di_Francesco-Guitter:Planar_maps_as_labeled_mobiles} is that it also applies to non-bipartite maps (but it yields a `three-type' tree even more complicated to study) so in principle, one may use it to prove the convergence of such maps, whereas the bijection from~\cite{Janson-Stefansson:Scaling_limits_of_random_planar_maps_with_a_unique_large_face} only applies to bipartite maps. On the other hand, the latter bijection reduces the technical analysis of the tree, which opens the possibility to study more general models of random bipartite maps, such as those from~\cite{Marzouk:Scaling_limits_of_random_bipartite_planar_maps_with_a_prescribed_degree_sequence} in more complicated `large faces' regimes. In particular, our proof does not necessitate a tight control on the geometry of the tree, since it mostly relies on its {\L}ukasiewicz path which is rather simple to study.

The rest of this paper is organised as follows: In Section~\ref{sec:cartes_arbres}, we recall the key bijection with labelled trees which in our case are randomly labelled size-conditioned Bienaymé--Galton--Watson trees. We recall their continuous analogues in Section~\ref{sec:arbres_continus} and state and prove their convergence in Theorem~\ref{thm:cv_serpents_cartes} in Section~\ref{sec:limite_arbres_etiquetes} which contains most of the technical parts and novelties of this work. Finally, in Section~\ref{sec:limite_cartes}, we state and prove Theorem~\ref{thm:profil} on the profile of distances and then prove Theorem~\ref{thm:cv_cartes}.

\section*{Acknowledgement}
I wish to thank Nicolas Curien for providing the simulations in Figure~\ref{fig:simu_cartes_stables} and for a discussion on what `discrete stable map' could mean.

This work was supported by a public grant as part of the Fondation Mathématique Jacques Hadamard.

\section{Maps and labelled trees}
\label{sec:cartes_arbres}

In this first section, let us briefly recall the notion of labelled (plane) trees and introduce some useful notation. We also describe the bijection between a pointed planar map and such a tree.

\subsection{Plane trees}
\label{sec:arbres}

Following the notation of Neveu~\cite{Neuveu:Arbres_et_processus_de_Galton_Watson}, we view discrete trees as words. Let $\N = \{1, 2, \dots\}$ be the set of all positive integers and set $\N^0 = \{\varnothing\}$. Then a (plane) \emph{tree} is a non-empty subset $T \subset \bigcup_{n \ge 0} \N^n$ such that:
\begin{enumerate}
\item $\varnothing \in T$;
\item if $u = (u_1, \dots, u_n) \in T$, then $pr(u) = (u_1, \dots, u_{n-1}) \in T$;
\item if $u = (u_1, \dots, u_n) \in T$, then there exists an integer $k_u \ge 0$ such that $ui = (u_1, \dots, u_n, i) \in T$ if and only if $1 \le i \le k_u$.
\end{enumerate}
We shall view each vertex $u$ of a tree $T$ as an individual of a population for which $T$ is the genealogical tree. The vertex $\varnothing$ is called the \emph{root} of the tree and for every $u = (u_1, \dots, u_n) \in T$, $pr(u) = (u_1, \dots, u_{n-1})$ is its \emph{parent}, $k_u$ is the number of \emph{children} of $u$ (if $k_u = 0$, then $u$ is called a \emph{leaf}, otherwise, $u$ is called an \emph{internal vertex}), and $u1, \dots, uk_u$ are these children from left to right, $\chi_u = u_n$ is the relative position of $u$ among its siblings, and $|u| = n$ is its \emph{generation}. We shall denote by $\llbracket u , v \rrbracket$ the unique non-crossing path between $u$ and $v$.

Fix a tree $T$ with $N+1$ vertices, listed $\varnothing = u_0 < u_1 < \dots < u_N$ in \emph{lexicographical order}. We describe two discrete paths which each encode $T$. First, its \emph{{\L}ukasiewicz path} $W = (W(j) ; 0 \le j \le N+1)$ is defined by $W(0) = 0$ and for every $0 \le j \le N$,
\[W(j+1) = W(j) + k_{u_j}-1.\]
One easily checks that $W(j) \ge 0$ for every $0 \le j \le N$ but $W(N+1)=-1$. Next, we define the \emph{height process} $H = (H(j); 0 \le j \le N)$ by setting for every $0 \le j \le N$,
\[H(j) = |u_j|.\]

The next lemma, whose proof is left as an exercise, gathers some deterministic results that we shall need (we refer to e.g. Le Gall~\cite{Le_Gall:Random_trees_and_applications} for a thorough discussion of such results). In order to simplify the notation, we identify the vertices of a tree with their index in the lexicographic order.

\begin{lem}\label{lem:codage_marche_Luka}
Let $T$ be a plane tree and $W$ be its {\L}ukasiewicz path. Fix a vertex $u \in T$, then
\[W(u k_u) = W(u),
\qquad
W(uj') = \inf_{[uj,uj']} W
\qquad\text{and}\qquad
j' - j = W(uj) - W(uj')\]
for every $1 \le j \le j' \le k_u$.
\end{lem}

Note that $W(u) - W(pr(u))$ equals the number of siblings of $u$ which lie to its right, so $W(u)$ equals the total number of individuals branching off to the right of the ancestral line $\llbracket \varnothing, u\llbracket$.

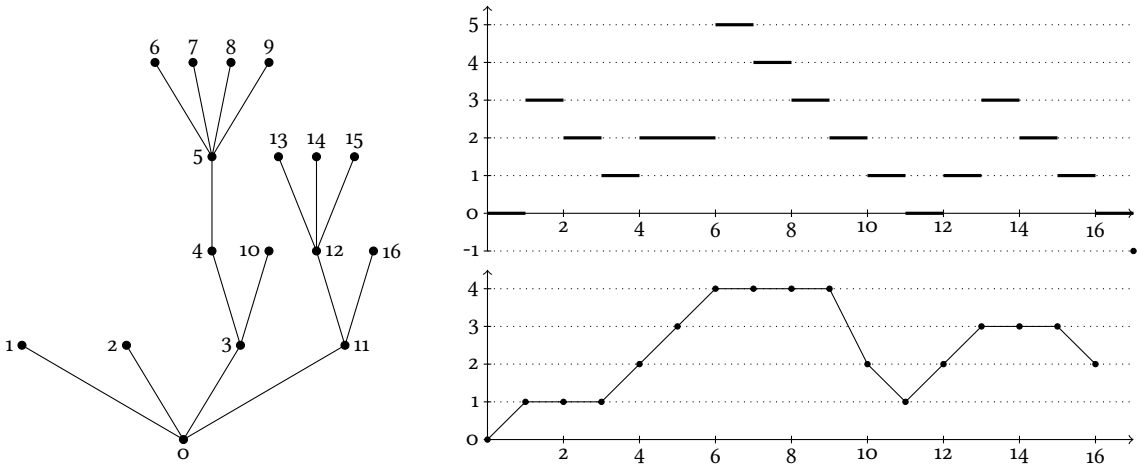
\begin{figure}[!ht]\centering
\def\r{.6}
\def\longueur{2.5}
\begin{footnotesize}
\begin{tikzpicture}[scale=.5]
\coordinate (1) at (0,0*\longueur);
	\coordinate (2) at (-4.25,1*\longueur);
	\coordinate (3) at (-1.5,1*\longueur);
	\coordinate (4) at (1.5,1*\longueur);
		\coordinate (5) at (.75,2*\longueur);
			\coordinate (6) at (.75,3*\longueur);
				\coordinate (7) at (-.75,4*\longueur);
				\coordinate (8) at (.25,4*\longueur);
				\coordinate (9) at (1.25,4*\longueur);
				\coordinate (10) at (2.25,4*\longueur);
		\coordinate (11) at (2.25,2*\longueur);
	\coordinate (12) at (4.25,1*\longueur);
		\coordinate (13) at (3.5,2*\longueur);
			\coordinate (14) at (2.5,3*\longueur);
			\coordinate (15) at (3.5,3*\longueur);
			\coordinate (16) at (4.5,3*\longueur);
		\coordinate (17) at (5,2*\longueur);

\draw
	(1) -- (2)	(1) -- (3)	(1) -- (4)	(1) -- (12)
	(4) -- (5)	(4) -- (11)
	(5) -- (6)
	(6) -- (7)	(6) -- (8)	(6) -- (9)	(6) -- (10)
	(12) -- (13)	(12) -- (17)
	(13) -- (14)	(13) -- (15)	(13) -- (16)
;

\foreach \x in {1, 2, ..., 17}
\draw[fill=black] (\x) circle (3pt);

\draw
	(1) node[below] {0}
	(2) node[left] {1}
	(3) node[left] {2}
	(4) node[left] {3}
	(5) node[left] {4}
	(6) node[left] {5}
	(7) node[above] {6}
	(8) node[above] {7}
	(9) node[above] {8}
	(10) node[above] {9}
	(11) node[left] {10}
	(12) node[right] {11}
	(13) node[right] {12}
	(14) node[above] {13}
	(15) node[above] {14}
	(16) node[above] {15}
	(17) node[right] {16};
\begin{scope}[shift={(8,6)}]
\draw[thin, ->]	(0,0) -- (17,0);
\draw[thin, ->]	(0,-1) -- (0,5.5);
\foreach \x in {-1, 1, 2, 3, 4, 5}
	\draw[dotted]	(0,\x) -- (17,\x);
\foreach \x in {-1, 0, 1, ..., 5}
	\draw (.1,\x)--(-.1,\x)	(0,\x) node[left] {\x};
\foreach \x in {2, 4, ..., 16}
	\draw (\x,.1)--(\x,-.1)	(\x,0) node[below] {\x};

\draw[very thick]
	(0, 0) -- ++ (1,0)
	++(0,3) -- ++ (1,0)
	++(0,-1) -- ++ (1,0)
	++(0,-1) -- ++ (1,0)
	++(0,1) -- ++ (1,0)
	++(0,0) -- ++ (1,0)
	++(0,3) -- ++ (1,0)
	++(0,-1) -- ++ (1,0)
	++(0,-1) -- ++ (1,0)
	++(0,-1) -- ++ (1,0)
	++(0,-1) -- ++ (1,0)
	++(0,-1) -- ++ (1,0)
	++(0,1) -- ++ (1,0)
	++(0,2) -- ++ (1,0)
	++(0,-1) -- ++ (1,0)
	++(0,-1) -- ++ (1,0)
	++(0,-1) -- ++ (1,0)
;
\draw[fill=black] (17,-1) circle (2pt);
\end{scope}
\begin{scope}[shift={(8,0)}]
\draw[thin, ->]	(0,0) -- (17,0);
\draw[thin, ->]	(0,0) -- (0,4.5);
\foreach \x in {1, 2, ..., 4}
	\draw[dotted]	(0,\x) -- (17,\x);
\foreach \x in {0, 1, ..., 4}
	\draw (.1,\x)--(-.1,\x)	(0,\x) node[left] {\x};
\foreach \x in {2, 4, ..., 16}
	\draw (\x,.1)--(\x,-.1)	(\x,0) node[below] {\x};

\draw[fill=black]
	(0, 0) circle (2pt) --
	++ (1, 1) circle (2pt) --
	++ (1, 0) circle (2pt) --
	++ (1, 0) circle (2pt) --
	++ (1, 1) circle (2pt) --
	++ (1, 1) circle (2pt) --
	++ (1, 1) circle (2pt) --
	++ (1, 0) circle (2pt) --
	++ (1, 0) circle (2pt) --
	++ (1, 0) circle (2pt) --
	++ (1, -2) circle (2pt) --
	++ (1, -1) circle (2pt) --
	++ (1, 1) circle (2pt) --
	++ (1, 1) circle (2pt) --
	++ (1, 0) circle (2pt) --
	++ (1, 0) circle (2pt) --
	++ (1, -1) circle (2pt)
;
\end{scope}
\end{tikzpicture}
\end{footnotesize}
\caption{A tree on the left, with its vertices listed in lexicographical order, and on the right, its {\L}ukasiewicz path $W$ on top and its height process $H$ below.}
\label{fig:codage_arbre}
\end{figure}

\subsection{Labelled trees}

For every $k \ge 1$, let us consider the set of \emph{bridges with no negative jumps}
\begin{equation}\label{eq:pont_sans_saut_negatif}
\mathcal{B}_k^+ = \left\{(x_1, \dots, x_k): x_1, x_2-x_1, \dots, x_k-x_{k-1} \in \{-1, 0, 1, 2, \dots\} \text{ and } x_k=0\right\}.
\end{equation}
A \emph{labelling} $\ell$ of a plane tree $T$ is a function defined on its set of vertices to $\Z$ such that
\begin{enumerate}
\item the root of $T$ has label $\ell(\varnothing) = 0$,
\item for every vertex $u$, with $k_u \ge 1$ children, the sequence of increments $(l(u1)-l(u), \dots, l(uk_u)-l(u))$ belongs to $\mathcal{B}_{k_u}^+$.
\end{enumerate}
We stress that the last child of every internal vertex carries the same label as its parent, for example, the right-most branch in the tree only contains zeros. 
Define the \emph{label process} $L(k) = \ell(u_k)$, where $(u_0, \dots, u_N)$ is the sequence of vertices of $T$ in lexicographical order; the labelled tree is encoded by the pair $(H, L)$, see Figure~\ref{fig:arbre_etiquete}.

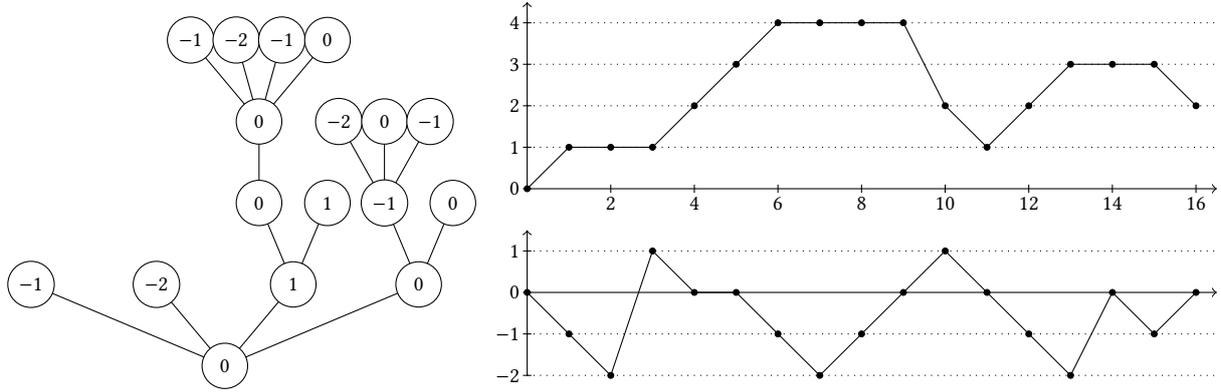
\begin{figure}[!ht]\centering
\def\r{.6}
\def\longueur{1.8}
\begin{tikzpicture}[scale=.6]
\coordinate (1) at (0,0*\longueur);
	\coordinate (2) at (-4.25,1*\longueur);
	\coordinate (3) at (-1.5,1*\longueur);
	\coordinate (4) at (1.5,1*\longueur);
		\coordinate (5) at (.75,2*\longueur);
			\coordinate (6) at (.75,3*\longueur);
				\coordinate (7) at (-.75,4*\longueur);
				\coordinate (8) at (.25,4*\longueur);
				\coordinate (9) at (1.25,4*\longueur);
				\coordinate (10) at (2.25,4*\longueur);
		\coordinate (11) at (2.25,2*\longueur);
	\coordinate (12) at (4.25,1*\longueur);
		\coordinate (13) at (3.5,2*\longueur);
			\coordinate (14) at (2.5,3*\longueur);
			\coordinate (15) at (3.5,3*\longueur);
			\coordinate (16) at (4.5,3*\longueur);
		\coordinate (17) at (5,2*\longueur);

\draw
	(1) -- (2)	(1) -- (3)	(1) -- (4)	(1) -- (12)
	(4) -- (5)	(4) -- (11)
	(5) -- (6)
	(6) -- (7)	(6) -- (8)	(6) -- (9)	(6) -- (10)
	(12) -- (13)	(12) -- (17)
	(13) -- (14)	(13) -- (15)	(13) -- (16)
;

\begin{scriptsize}
\node[circle, minimum size=\r cm, fill=white, draw] at (2) {$-1$};
\node[circle, minimum size=\r cm, fill=white, draw] at (3) {$-2$};
\node[circle, minimum size=\r cm, fill=white, draw] at (7) {$-1$};
\node[circle, minimum size=\r cm, fill=white, draw] at (8) {$-2$};
\node[circle, minimum size=\r cm, fill=white, draw] at (9) {$-1$};
\node[circle, minimum size=\r cm, fill=white, draw] at (10) {$0$};
\node[circle, minimum size=\r cm, fill=white, draw] at (11) {$1$};
\node[circle, minimum size=\r cm, fill=white, draw] at (14) {$-2$};
\node[circle, minimum size=\r cm, fill=white, draw] at (15) {$0$};
\node[circle, minimum size=\r cm, fill=white, draw] at (16) {$-1$};
\node[circle, minimum size=\r cm, fill=white, draw] at (17) {$0$};

\node[circle, minimum size=\r cm, fill=white, draw] at (1) {$0$};
\node[circle, minimum size=\r cm, fill=white, draw] at (4) {$1$};
\node[circle, minimum size=\r cm, fill=white, draw] at (5) {$0$};
\node[circle, minimum size=\r cm, fill=white, draw] at (6) {$0$};
\node[circle, minimum size=\r cm, fill=white, draw] at (12) {$0$};
\node[circle, minimum size=\r cm, fill=white, draw] at (13) {$-1$};
\end{scriptsize}
\end{tikzpicture}
\begin{scriptsize}
\begin{tikzpicture}[scale=.55]
\draw[thin, ->]	(0,0) -- (16.5,0);
\draw[thin, ->]	(0,0) -- (0,4.5);
\foreach \x in {1, 2, ..., 4}
	\draw[dotted]	(0,\x) -- (16.5,\x);
\foreach \x in {0, 1, ..., 4}
	\draw (.1,\x)--(-.1,\x)	(0,\x) node[left] {$\x$};
\foreach \x in {2, 4, ..., 16}
	\draw (\x,.1)--(\x,-.1)	(\x,0) node[below] {$\x$};
\coordinate (0) at (0, 0);
\coordinate (1) at (1, 1);
\coordinate (2) at (2, 1);
\coordinate (3) at (3, 1);
\coordinate (4) at (4, 2);
\coordinate (5) at (5, 3);
\coordinate (6) at (6, 4);
\coordinate (7) at (7, 4);
\coordinate (8) at (8, 4);
\coordinate (9) at (9, 4);
\coordinate (10) at (10, 2);
\coordinate (11) at (11, 1);
\coordinate (12) at (12, 2);
\coordinate (13) at (13, 3);
\coordinate (14) at (14, 3);
\coordinate (15) at (15, 3);
\coordinate (16) at (16, 2);
\newcommand{\lastx}{0}
\foreach \x [remember=\x as \lastx] in {1, 2, 3, ..., 16} \draw (\lastx) -- (\x);

\foreach \x in {0, 1, 2, 3, ..., 16} \draw [fill=black] (\x)	circle (2pt);
\begin{scope}[shift={(0,-2.5)}]
\draw[thin, ->]	(0,0) -- (16.5,0);
\draw[thin, ->]	(0,-2) -- (0,1.5);
\foreach \x in {-2, -1, 1}
	\draw[dotted]	(0,\x) -- (16.5,\x);
\foreach \x in {-2, -1, 0, 1}
	\draw (.1,\x)--(-.1,\x)	(0,\x) node[left] {$\x$};
%
\coordinate (0) at (0, 0);
\coordinate (1) at (1, -1);
\coordinate (2) at (2, -2);
\coordinate (3) at (3, 1);
\coordinate (4) at (4, 0);
\coordinate (5) at (5, 0);
\coordinate (6) at (6, -1);
\coordinate (7) at (7, -2);
\coordinate (8) at (8, -1);
\coordinate (9) at (9, 0);
\coordinate (10) at (10, 1);
\coordinate (11) at (11, 0);
\coordinate (12) at (12, -1);
\coordinate (13) at (13, -2);
\coordinate (14) at (14, 0);
\coordinate (15) at (15, -1);
\coordinate (16) at (16, 0);
\renewcommand{\lastx}{0}
\foreach \x [remember=\x as \lastx] in {1, 2, 3, ..., 16} \draw (\lastx) -- (\x);

\foreach \x in {0, 1, 2, 3, ..., 16} \draw [fill=black] (\x)	circle (2pt);
\end{scope}
\end{tikzpicture}
\end{scriptsize}
\caption{A labelled tree on the left, and on the right, its height process on top and its label height process below.}
\label{fig:arbre_etiquete}
\end{figure}

Without further notice, throughout this work, every {\L}ukasiewicz path shall be viewed as a step function, jumping at integer times, whereas height and label processes shall be viewed as continuous functions after interpolating linearly between integer times.

\subsection{Labelled trees and pointed maps}
\label{sec:bijection}

Bouttier, Di Francesco \& Guitter~\cite{Bouttier-Di_Francesco-Guitter:Planar_maps_as_labeled_mobiles} proved that pointed maps are in bijection with some labelled trees, different from the preceding section; in the bipartite case, Janson \& Stef{\'a}nsson~\cite{Janson-Stefansson:Scaling_limits_of_random_planar_maps_with_a_unique_large_face} then related these trees to labelled trees as in the preceding section. Let us describe a direct construction of this bijection between labelled trees and pointed maps and leave to the reader as an exercise to verify that it indeeds corresponds to the two previous bijections (one may compare the figures here and those in~\cite{Marzouk:Scaling_limits_of_random_bipartite_planar_maps_with_a_prescribed_degree_sequence}).

Let us start with the construction of a pointed map from a labelled tree $(T, \ell)$, depicted in Figure~\ref{fig:arbre_carte}; the construction contains two steps. Let $(u_0, \dots, u_N)$ be the vertices of $T$ listed in lexicographical order. For every $0 \le i \le N$, set $u_{N+1+i} = u_i$. We add an extra vertex $\star$ labelled $\min_{u \in T} \ell(u)-1$ outside of the tree $T$ and construct a first planar graph $G$ on the vertex-set of $T$ and $\star$ by drawing edges as follows: for every $0 \le i \le N-1$,
\begin{itemize}
\item if $\ell(u_i) > \min_{0 \le k \le N} \ell(u_k)$, then we draw an edge between $u_i$ and $u_j$ where $j = \min\{k > i: \ell(u_k) = \ell(u_i)-1\}$,
\item if $\ell(u_i) = \min_{0 \le k \le N} \ell(u_k)$, then we draw an edge between $u_i$ and $\star$.
\end{itemize}
We stress that we exclude the last vertex $u_N$ in this construction; it indeed yields a planar graph $G$. In a second step, we merge every internal vertex of the tree $T$ with their last child; then $G$ becomes a map $M$ with labelled vertices. We shift all labels by subtracting $\min_{u \in T} \ell(u)-1$;
it can be checked that these new labels are just the graph distance to $\star$ in the map $M$. We also distinguish the image after the merging operation of the first edge that we drew, for $i=0$. The latter is non-oriented; let $e_+$ and $e_-$ be its extremities so that $\dgr(e_-, \star) = \dgr(e_+, \star)-1$ and let us orient the edge from $e_+$ to $e_-$; these maps are called \emph{negative} in~\cite{Marckert-Miermont:Invariance_principles_for_random_bipartite_planar_maps}.

\begin{figure}[!ht] \centering
\includegraphics[height=9\baselineskip, page = 2]{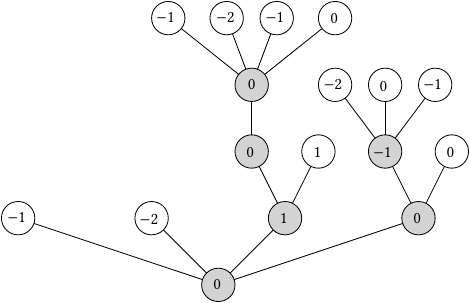}
\qquad
\includegraphics[height=10.5\baselineskip, page = 3]{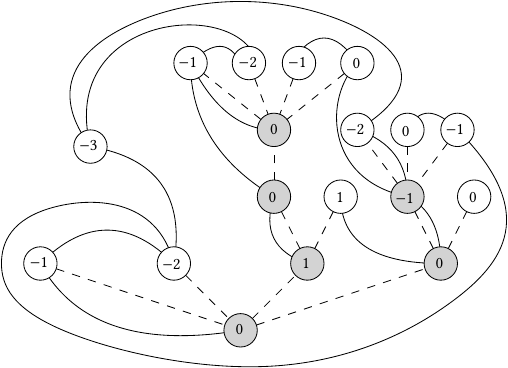}
\newline~\newline
\includegraphics[height=10.5\baselineskip, page = 4]{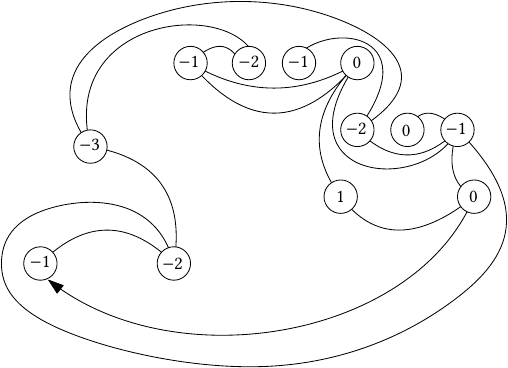}
\qquad
\includegraphics[height=10.5\baselineskip, page = 5]{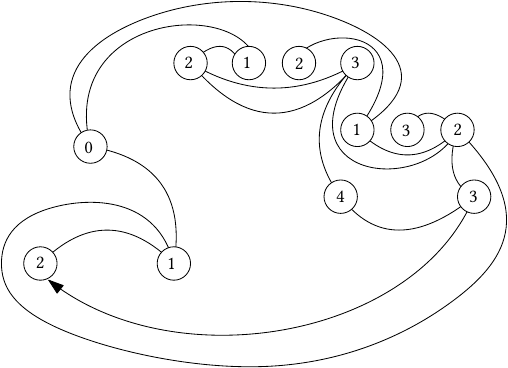}
\caption{The negative map associated with a labelled tree.}
\label{fig:arbre_carte}
\end{figure}

\begin{figure}[!ht] \centering
\includegraphics[height=10.5\baselineskip, page = 6]{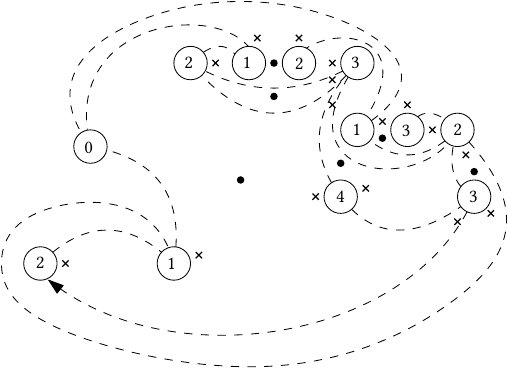}
\qquad
\includegraphics[height=8\baselineskip, page = 7]{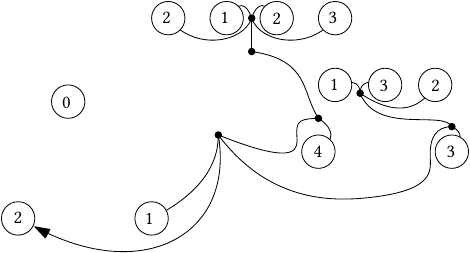}
\newline~\newline
\includegraphics[height=8\baselineskip, page = 9]{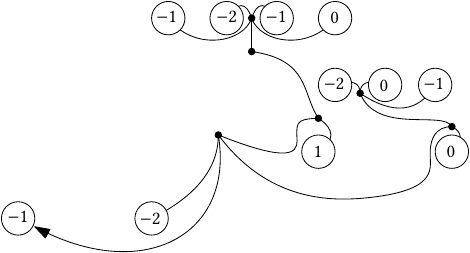}
\qquad
\includegraphics[height=8\baselineskip, page = 10]{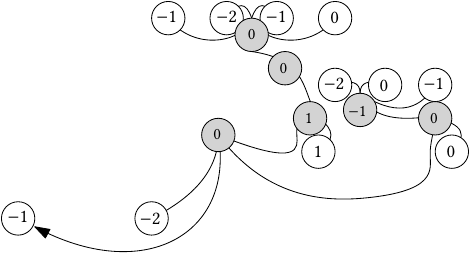}
\caption{The labelled tree associated with a negative map.}
\label{fig:carte_arbre}
\end{figure}

Let us next construct a labelled tree from a negative pointed map $(M, \star)$, as depicted in Figure~\ref{fig:carte_arbre}. First, label all vertices by their graph distance to $\star$. In every face of $M$, place a new unlabelled vertex and mark each corner when the next vertex of $M$ in clockwise order has a smaller label.
Then start with the root-face, adjacent to the right of the root-edge. Link the new vertex in this face to every marked corner if it is the only marked corner of this vertex, otherwise, erase the mark and link the new vertex to the one in the face which contains the next marked corner of this vertex in clockwise order. Proceed similarly with the new vertices attached to the one in the root-face: link each of them to the marked corners in their face if they are the only remaining ones around their vertex, otherwise, remove the mark and link the new vertex to the next one in clockwise order around the vertex. Continue recursively until all faces have been considered. This yields a planar tree that we root at the new vertex in the root-face, whose first child is either $e_-$ the target of the root-edge or the new vertex in the next face in clockwise order around it if any. Then assign to each new vertex the label of its last child and finally shift all labels so the root of the tree has label $0$ to get a labelled tree as in the preceding section.

We claim that these constructions are the inverse of one another and yield a bijection between labelled trees and negative maps (the construction is very close to~\cite{Bouttier-Di_Francesco-Guitter:Planar_maps_as_labeled_mobiles}, one can thus follow their detailed proof). Recall that the root-face of a map is the face adjacent to the right of the root-edge. This bijection enjoys the following properties:
\begin{enumerate}
\item The leaves of the tree are in one to one correspondence with the vertices different from the distinguished one in the map, and the label of a leaf minus the infimum over all labels, plus one, equals the graph distance between the corresponding vertex of the map and the distinguished vertex.
\item The internal vertices of the tree are in one to one correspondence with the faces of the map, and the number of children of the vertex is half the degree of the face.
\item The root-face of the map corresponds to the root-vertex of the tree.
\item The number of edges of the map and the tree are equal.
\end{enumerate}
In order to have a bijection between labelled trees and \emph{positive} maps (in which the root-edge is oriented from $e_-$ to $e_+$), one just reverse the root-edge in order to get a negative map. Note that Property~(iii) above does not hold anymore and it does not seem clear which internal vertex of the tree corresponds to the original root-face. Nonetheless, by `mirror symmetry' of the map (which preserves positivity or negativity of the map), the degree of the faces on both sides of the root-edge have the same distribution, so in both cases of positive or negative maps, the half-degree distribution of the root-face is the offspring distribution of the root of the tree.

Property~(i) above explains how to partially translate the metric properties of the map to the labelled tree, whereas Property~(ii) is important because it gives us the distribution of the tree when the map is a random Boltzmann map, as described below.

\subsection{Random labelled trees}
\label{sec:BGW}

Let us introduce the law of the labelled tree associated with a pointed map sampled from $\P^{\q, \bullet}$.
Let $q_0 = 1$ and define the power series
\[g_\q(x) = \sum_{k \ge 0} x^k \binom{2k-1}{k-1} q_k, \qquad x \ge 0.\]
Then $g_\q$ is convex, strictly increasing and continuous until its radius of convergence, and $g_\q(0)=1$. In particular, it has at most two fixed points, and if it has exactly one, then at that point, the graph of $g_\q$ either crosses the line $y=x$, or is tangent to it. 
It was argued in~\cite[Section~7.1]{Marzouk:Scaling_limits_of_random_bipartite_planar_maps_with_a_prescribed_degree_sequence}, recasting the discussion from~\cite[Section~1.2]{Marckert-Miermont:Invariance_principles_for_random_bipartite_planar_maps} in the context of the Janson--Stef{\'a}nsson bijection, that the sequence $\q$ is admissible and critical exactly when $g_\q$ falls into the last case, and we denote by $Z_\q$ the only fixed point, which satisfies $g_\q'(Z_\q) = 1$. Let us mention that $Z_\q$ equals $(W^{\q,\bullet}+1)/2 > 1$. Such a sequence $\q$ thus induces a probability measure on $\Z_+ = \{0, 1, 2, \dots\}$ with mean one, given by:
\begin{equation}\label{eq:loi_GW_carte_Boltzmann}
\mu_\q(k) = Z_\q^{k-1} \binom{2k-1}{k-1} q_k,
\qquad k \ge 0.
\end{equation}

We shall consider random labelled trees, sampled as follows. First, let $T$ be a Bienaymé--Galton--Watson tree with offspring distribution $\mu_\q$, which means that the probability that $T$ equals a given finite tree $\tau$ is $\prod_{u \in \tau} \mu_\q(k_u)$. For every subset $A \subset \Z_+$ such that $\mu_\q(A) \ne 0$ and for every $n \ge 1$, we let $T_{A,n}$ be such a tree conditioned to have exactly $n$ vertices with offspring in $A$; the asymptotic behaviour of such trees has been investigated by Kortchemski~\cite{Kortchemski:Invariance_principles_for_Galton_Watson_trees_conditioned_on_the_number_of_leaves}, with the restriction that either $A$ or its complement is finite. We shall be particularly interested in the sets $A=\Z_+$ so the tree is conditioned on its total progeny, $A=\{0\}$ so the tree is conditioned on its number of leaves, and $A = \N$ so the tree is conditioned on its number of internal vertices. We let $\zeta(T_{A,n})$ be the number of edges of $T_{A,n}$.

Next, conditional on the tree $T$ (or $T_{A,n}$), we sample uniformly random labels $(\ell(u))_{u \in T}$ satisfying the conditions described in Section~\ref{sec:arbres}: the root has label $\ell(\varnothing) = 0$ and the sequences $(\ell(ui)-\ell(u))_{1 \le i \le k_u}$ are independent when $u$ ranges over all internal vertices of $T$ and are distributed respectively uniformly at random in $\mathcal{B}_{k_u}^+$. Let us observe that the cardinal of $\mathcal{B}_k^+$ is precisely the binomial factor $\binom{2k-1}{k-1}$ in the definition of $\mu_\q$. Also, it is well-known and easy to check that a uniform random bridge in $\mathcal{B}^+_k$ has the law of the first $k$ steps of a random walk conditioned to end at $0$, with step distribution $\sum_{i \ge -1} 2^{-i-2} \delta_i$, which is centred and with variance $2$.

One easily checks (see e.g.~\cite[Proposition~11]{Marzouk:Scaling_limits_of_random_bipartite_planar_maps_with_a_prescribed_degree_sequence}) that this labelled tree $(T, (\ell(u))_{u \in T})$ is the one associated, in the bijection described previously, with a pointed Boltzmann map sampled from $\P^{\q, \bullet}$. 
Finally, the tree is conditioned to have $n$ vertices, or $n$ internal vertices, or $n$ leaves, when the map is conditioned to have $n-1$ edges, $n$ faces, and $n+1$ vertices respectively.
Thanks to Property~(iii) of the bijection in the preceding section, $\mu_\q$ is the law of the half-degree of the root-face under $\P^{\q, \bullet}$. For the rest of this paper, we further assume that it belongs to the domain of attraction of a stable law with index $\alpha \in (1,2]$, which means that either it has finite variance $\sum_{k=0}^\infty k^2 \mu_\q(k) < \infty$ and then $\alpha = 2$, or the tail can be written as $\sum_{k=j}^\infty \mu_\q(k) = j^{-\alpha} l(j)$, where $l$ is a \emph{slowly varying} function at infinity which means that for every $c > 0$, it holds that $\lim_{x \to \infty} l(cx) / l(x) = 1$.

We refer the reader to~\cite[Proposition~4]{Curien-Richier:Duality_of_random_planar_maps_via_percolation} for three equivalent assumptions. Using the notation from this reference, the root-face under $\P^\q$ has degree $2k$ with probability proportional to $q_k W_\q^{(k)}$ which, under our assumption, behaves as $q_k r_\q^{-k} k^{-\alpha-1/2}$. Using the fact that $q_k r_\q^{-k}$ is almost $\nu_\q(k)$ and that $\nu_\q$ has regularly varying tails with index $\alpha-1/2$, we see that, informally, $\nu_\q(k)$ behaves as $k^{-\alpha-1/2}$, so finally $q_k W_\q^{(k)} \approx k^{-2\alpha-1}$ and so the root-face under $\P^\q$ has regularly varying tails with index $2\alpha > 2$. This can be made rigorous using similar arguments to~\cite[Proposition~4]{Curien-Richier:Duality_of_random_planar_maps_via_percolation}.

\section{Continuous labelled trees}
\label{sec:arbres_continus}

In this short section, we briefly describe the continuous limits of labelled size-conditioned Bienaymé--Galton--Watson trees, the statement and proof of such a convergence are given in Section~\ref{sec:limite_arbres_etiquetes}.

For the rest of this paper, we fix an admissible and critical sequence $\q$ such that its support generates the whole group $\Z$ and such that $\mu_\q$ defined by~\eqref{eq:loi_GW_carte_Boltzmann} belongs to the domain of attraction of a stable law with index $\alpha \in (1,2]$. Then there exists an increasing sequence $(B_n)_{n \ge 1}$ such that if $(\xi_n)_{n \ge 1}$ is a sequence of i.i.d. random variables sampled from $\mu_\q$, then $B_n^{-1} (\xi_1 + \dots + \xi_n - n)$ converges in distribution to a random variable $X^{(\alpha)}$ whose law is given by the Laplace exponent $\E[\exp(-\lambda X^{(\alpha)})] = \exp(\lambda^\alpha)$ for every $\lambda \ge 0$. Recall that $n^{-1/\alpha} B_n$ is slowly varying at infinity and that if $\mu_\q$ has variance $\sigma^2_\q \in (0,\infty)$, then this falls in the case $\alpha=2$ and we may take $B_n = (n \sigma^2_\q/2)^{1/2}$. We stress that with this normalisation, $X^{(2)}$ has the centred Gaussian law with variance $2$.

\subsection{The stable trees}
\label{sec:arbres_stables}

The continuous analog of size-conditioned Bienaymé--Galton--Watson trees are the so-called \emph{stable Lévy trees} with index $\alpha \in (1,2]$. Let $\Xexc = (\Xexc_t ; t \in [0,1])$ denote the normalised excursion of the $\alpha$-stable \emph{Lévy process} with no negative jump, whose value at time $1$ has the law of $X^{(\alpha)}$, and let further $\Hexc = (\Hexc_t ; t \in [0,1])$ be the associated \emph{height function}; we refer to e.g.~\cite{Duquesne:A_limit_theorem_for_the_contour_process_of_conditioned_Galton_Watson_trees} for the definitions of this object.
In the case $\alpha=2$, the two processes $\Xexc$ and $\Hexc$ are equal, both to $\sqrt{2}$ times the standard Brownian excursion. In any case, $\Hexc$ is a non-negative, continuous function, which vanishes only at $0$ and $1$. As any such function, it encodes a `continuous tree' called the $\alpha$-stable Lévy tree $\CRT$ of Duquesne, Le Gall \& Le Jan~\cite{Duquesne:A_limit_theorem_for_the_contour_process_of_conditioned_Galton_Watson_trees,Le_Gall-Le_Jan:Branching_processes_in_Levy_processes_the_exploration_process}, which generalises the celebrated Brownian tree of Aldous~\cite{Aldous:The_continuum_random_tree_3} in the case $\alpha=2$. Precisely, for every $s, t \in [0,1]$, set
\[d(s,t) = \Hexc_s + \Hexc_t - 2 \min_{r \in [\min(s, t), \max(s, t)]} \Hexc_r.\]
One easily checks that $d$ is a random pseudo-metric on $[0,1]$, we then define an equivalence relation on $[0,1]$ by setting $s \sim t$ whenever $d(s,t)=0$. Consider the quotient space $\CRT = [0,1] / \sim$, we let $\pi$ be the canonical projection $[0,1] \to \CRT$; then $d$ induces a metric on $\CRT$ that we still denote by $d$. The space $(\CRT, d)$ is a so-called compact real-tree, naturally rooted at $\pi(0) = \pi(1)$.

\subsection{The continuous distance process}
We construct next another process $\Lab = (\Lab_t ; t \in [0,1])$ called the \emph{continuous distance process} on the same probability space as $\Hexc$ which is intrinsically different according as wether $\alpha=2$ or $\alpha<2$. Let us start with the latter case which is analogous to the discrete setting. Indeed, in the discrete setting, the label increment between a vertex and is parent was given by the value of a random discrete bridge of length equal to the offspring of the parent, at a time given by the position of the child. Loosely speaking, we do the same when $\alpha < 2$, by taking random Brownian bridges.

Precisely, suppose that $\alpha < 2$ and that $(b_i)_{i \ge 1}$ are i.d.d. standard Brownian bridges of duration $1$ from $0$ to $0$, defined on the same probability space as $\Xexc$ and independent of the latter; by the scaling property, for every $x > 0$, the process $(x^{1/2} b_1(t/x) ; t \in [0,x])$ is a standard Brownian bridge of duration $x$. For every $0 \le s \le t \le 1$, put
\[\mathscr{I}_{s, t} = \inf_{r \in [s,t]} \Xexc_r.\]
For every $t \in (0,1)$ let $\Delta \Xexc_t = \Xexc_t - \Xexc_{t-} \ge 0$ be the `jump' of $\Xexc$ at time $t$ and let $(t_i)_{i \ge 1}$ be a measurable enumeration of those times $t$ such that $\Delta \Xexc_t > 0$. We then put for every $t \in [0,1]$:
\begin{equation}\label{eq:etiquettes_stables}
\Lab_t = \sqrt{2} \sum_{i \ge 1} \Delta \Xexc_{t_i}^{1/2} b_i\left(\frac{\mathscr{I}_{t_i, t} - \Xexc_{t-}}{\Delta \Xexc_{t_i}}\right) \ind{\mathscr{I}_{t_i, t} \ge \Xexc_{t-}} \ind{t_i \le t}.
\end{equation}
According to Le Gall \& Miermont~\cite[Proposition~5 and~6]{Le_Gall-Miermont:Scaling_limits_of_random_planar_maps_with_large_faces}, this series converges in $L^2$ and the process $\Lab$ admits a continuous modification, even H\"{o}lder continuous for any index smaller than $1/(2\alpha)$. The factor $\sqrt{2}$ is added here in the definition of $\Lab$ in order to have statements without constants.

When $\alpha=2$, the process $\Xexc$ is $\sqrt{2}$ times the Brownian excursion so it has continuous paths. To understand the definition, imagine that in the discrete setting, the tree $T_n$ is binary: internal vertices always have two children, then the label increment between such an internal vertex and its first child equals $-1$ or $1$ with probability $1/2$ each, and given a `typical' vertex, each of its ancestor is either the first or the second child of its parent, with probability roughly $1/2$ each, so the sequence of increments along an ancestral line resembles a centred random walk with step $-1$ or $1$ with probability $1/4$ each and $0$ with probability $1/2$. In the continuous setting of the Brownian tree, we define the process $\Lab$ conditional on $\Hexc$ as a centred Gaussian process satisfying for every $s,t \in [0,1]$,
\[\Esc{|\Lab_s - \Lab_t|^2}{\Hexc} = \frac{2}{3} \cdot d(s,t)
\quad\text{or, equivalently,}\quad
\Esc{\Lab_s \Lab_t}{\Hexc} = \frac{2}{3} \min_{r \in [\min(s, t), \max(s, t)]} \Hexc_r.\]
Again, the factor $2/3$ removes the constants in our statements and will be explained below.
This process is called the \emph{head of Brownian snake} driven by $\Hexc$~\cite{Le_Gall:Nachdiplomsvorlesung,Duquesne-Le_Gall:Random_trees_Levy_processes_and_spatial_branching_processes}; it is known, see, e.g.~\cite[Chapter~IV.4]{Le_Gall:Nachdiplomsvorlesung} that it admits a continuous version.

In all cases $\alpha \in (1,2]$, without further notice, we shall work throughout this paper with the continuous version of $\Lab$. Observe that, almost surely, $\Lab_0=0$ and $\Lab_s = \Lab_t$ whenever $s \sim t$ so $\Lab$ can be seen as a random motion indexed by $\CRT$ by setting $\Lab_{\pi(t)} = \Lab_t$ for every $t \in [0,1]$. We interpret $\Lab_x$ as the label of an element $x \in \CRT$; the pair $(\CRT, (\Lab_x; x \in \CRT))$ is a continuous analog of labelled plane trees.

\begin{rem}
We point out that, when $\alpha=2$, the process $\Lab$ is loosely speaking (up to constants) a Brownian motion indexed by the Brownian tree, which is denoted by $\Snake$ in~\cite{Marzouk:Scaling_limits_of_discrete_snakes_with_stable_branching}, but it is \emph{not} a Brownian motion indexed by the stable tree in the case $\alpha < 2$, this object is studied in~\cite{Marzouk:Scaling_limits_of_discrete_snakes_with_stable_branching}.
\end{rem}

\section{Scaling limits of labelled trees}
\label{sec:limite_arbres_etiquetes}

Throughout this section, we fix $A \subset \Z_+$ such that either $A$ or $\Z_+\setminus A$ is finite and $\mu_\q(A) \ne 0$ and for every $n \ge 1$, we let $T_{A,n}$ be a Bienaymé--Galton--Watson tree with offspring distribution $\mu_\q$ conditioned to have exactly $n$ vertices with offspring in $A$; recall that $\zeta(T_{A,n})$ denotes the number of edges of $T_{A,n}$. We then sample uniformly random labels $(\ell(u))_{u \in T_{A,n}}$ as in Section~\ref{sec:BGW}.

Duquesne~\cite{Duquesne:A_limit_theorem_for_the_contour_process_of_conditioned_Galton_Watson_trees} in the case $A = \Z_+$ (so $\zeta(T_{A,n}) = n-1$), see also Kortchemski~\cite{Kortchemski:A_simple_proof_of_Duquesne_s_theorem_on_contour_processes_of_conditioned_Galton_Watson_trees}, and then Kortchemski~\cite{Kortchemski:Invariance_principles_for_Galton_Watson_trees_conditioned_on_the_number_of_leaves} in the general case, proved the convergence of the {\L}ukasiewicz path and height process:
\begin{equation}\label{eq:Duquesne_Kortchemski}
\left(\frac{1}{B_{\zeta(T_{A,n})}} W_n(\zeta(T_{A,n}) t), \frac{B_{\zeta(T_{A,n})}}{\zeta(T_{A,n})} H_n(\zeta(T_{A,n}) t)\right)_{t \in [0,1]}
\cvloi (\Xexc_t, \Hexc_t)_{t \in [0,1]},
\end{equation}
in $\mathscr{D}([0,1], \R) \otimes \mathscr{C}([0,1], \R)$. Let us point out that the work~\cite{Kortchemski:Invariance_principles_for_Galton_Watson_trees_conditioned_on_the_number_of_leaves} focuses on the case $A=\{0\}$ and most of the results we shall need are developed in this case, but as explained in Section~8 there, the arguments extend to the general case, at least as long as either $A$ or its complement is finite.

\begin{rem}\label{rem:nombre_aretes_arbre}
As observed by Kortchemski~\cite{Kortchemski:Invariance_principles_for_Galton_Watson_trees_conditioned_on_the_number_of_leaves}, see e.g. Corollary~3.3 there for a stronger result, it holds that
\[n^{-1} \zeta(T_{A,n}) \cvproba \mu_\q(A)^{-1}.\]
Therefore, we may replace $\zeta(T_{A,n})$ by $\mu_\q(A)^{-1} n$ in~\eqref{eq:Duquesne_Kortchemski} above and Theorem~\ref{thm:cv_serpents_cartes} below. In the cases $A = \{0\}$ and $A = \N$, recall that $T_{A,n}$ is related to a Boltzmann map conditioned to have $n+1$ vertices and $n$ faces respectively. Since $\mu_\q(0) = Z_\q^{-1}$, this explains Remark~\ref{rem:nombre_aretes_carte}.
\end{rem}

As alluded in the introduction, the key to prove Theorem~\ref{thm:cv_cartes} is the following result.

\begin{thm}
\label{thm:cv_serpents_cartes}
The convergence in distribution
\[\left(B_{\zeta(T_{A,n})}^{-1/2} L_n(\zeta(T_{A,n}) t)\right)_{t \in [0,1]}
\cvloi (\Lab_t)_{t \in [0,1]},\]
holds in $\mathscr{C}([0,1], \R)$ jointly with~\eqref{eq:Duquesne_Kortchemski}.
\end{thm}

The proof of the convergence of $L_n$ occupies the rest of this section. We first prove that it is tight and then we characterise the finite dimensional marginals, using two different arguments for the non-Gaussian case $\alpha < 2$ and the Gaussian case $\alpha = 2$, since the limit process $\Lab$ is defined in two different ways. Indeed, let us comment on this statement and on the constants in the definition of $\Lab$. We assume $A = \Z_+$ to ease the notation in this informal discussion.

For a vertex $u \in T_{A,n}$, the label increments between consecutive ancestors are independent and distributed as $X_{k,j}$ when an ancestor has $k \ge 1$ children and the one on the path to $u$ is the $j$-th one, where $(X_{k,1}, \dots, X_{k,k})$ is uniformly distributed in $\mathcal{B}_k^+$, as defined in \eqref{eq:pont_sans_saut_negatif}. Since the latter has the law of a random walk conditioned to be at $0$ at time $k$, with step distribution $\sum_{k \ge -1} 2^{-k-2} \delta_k$ which is centred and with variance $2$, then a conditional version of Donsker's invariance principle for random bridges (see e.g.~\cite[Lemma~10]{Bettinelli-Scaling_limits_for_random_quadrangulations_of_positive_genus} for a detailed proof of the latter) yields
\begin{equation}\label{eq:Donsker_ponts}
\left((2k)^{-1/2} X_{k, kt}\right)_{t \in [0,1]} \cvloi (b_t)_{t \in [0,1]},
\end{equation}
where, as usual, on the left we have linearly interpolated, and $b$ is the standard Brownian bridge. The factor $\sqrt{2}$ is the same as in the definition of $\Lab$ for $\alpha < 2$ in~\eqref{eq:etiquettes_stables} and one must check that the $k$'s and $j$'s converge towards the $\Delta \Xexc_{t_i}$'s and the $\mathscr{I}_{t_i, t} - \Xexc_{t-}$'s. 

In the case $\alpha=2$, suppose furthermore that the variance $\sigma_\q^2$ of $\mu_q$ is finite, so $B_n = (n \sigma_\q^2 / 2)^{1/2}$. Since $X_{k,j}$ has variance $2j(k-j)/(k+1)$ and, as we will see, there is typically a proportion about $\mu_\q(k)$ of such ancestors, then $\ell(u)$ has variance about
\[\sum_{k \ge 1} \sum_{j=1}^k |u| \mu_\q(k) \frac{2j(k-j)}{k+1}
= |u| \sum_{k \ge 1} \mu_\q(k) \frac{k(k-1)}{3}
\approx |u| \frac{\sigma_\q^2}{3}.\]
If $u$ is the vertex visited at time $\lfloor n t\rfloor$ in lexicographical order, then, by~\eqref{eq:Duquesne_Kortchemski} we have $|u| \approx (n/B_{n}) \Hexc_t = (2n /\sigma_\q^2)^{1/2}\Hexc_t$ so we expect $L_n(n t)$, once divided by $B_{n}^{1/2} = (n \sigma_\q^2 / 2)^{1/4}$, to be asymptotically Gaussian with variance
\[\left(\frac{2}{n \sigma_\q^2}\right)^{1/2} \left(\frac{2n}{\sigma_p^2}\right)^{1/2} \Hexc_t \frac{\sigma_p^2}{3}
= \frac{2}{3} \Hexc_t,\]
Which exactly corresponds to $\Lab_t$. The case $\alpha=2$ but $\mu_\q$ has infinite variance is more involved, but this sketch can be adapted, by taking the truncated variance.

\subsection{Tightness of the label process}
\label{sec:tension_labels}

The first step towards the proof of Theorem~\ref{thm:cv_serpents_cartes} is to show that the sequence of processes
\[\left(B_{\zeta(T_{A,n})}^{-1/2} L_n(\zeta(T_{A,n}) t)\right)_{t \in [0,1]}\]
is tight. This was proved in~\cite[Proposition 7]{Marzouk:Scaling_limits_of_random_bipartite_planar_maps_with_a_prescribed_degree_sequence} in a slightly different context of trees `with a prescribed degree sequence' in a finite variance regime but the argument are easily adapted to our case. 
The main point is to apply Kolmogorov's tightness criterion; thanks to the properties of uniform random bridges in $\mathcal{B}_k^+$, we can see that the increment of labels between a vertex $u$ and one of its ancestors $v$ is about the square-root of the numbers of vertices branching off of the path $\llbracket u, v \llbracket$, which can be described in terms of the {\L}ukasiewicz path. We thus shall need later the following tail bounds.

\begin{lem}\label{lem:moments_marche_Luka}
Fix any $\theta \in (0, 1/\alpha)$. There exists $c_1, c_2 > 0$ such that for every $n$ large enough, for every $0 \le s \le t \le 1$, every $x \ge 0$, and every $\delta \in (0, \alpha/(\alpha-1))$, we have
\[\Pr{W_n(\zeta(T_{A,n}) s) - \min_{s \le r \le t} W_n(\zeta(T_{A,n}) r) > B_{\zeta(T_{A,n})} |t-s|^\theta x} \le c_1 \exp(- c_2 x^\delta).\]
Consequently, the moments of $B_{\zeta(T_{A,n})}^{-1} |t-s|^{-\theta} (W_n(\zeta(T_{A,n}) s) - \min_{s \le r \le t} W_n(\zeta(T_{A,n}) r))$ are uniformly bounded.
\end{lem}

\begin{proof}
First note that we may restrict ourselves to times $|t-s| \le 1/2$. Let us start with the more familiar case $A=\Z_+$. It is well-known that $W_n$ is an excursion of a random walk $S$ with i.i.d. steps distributed as $\sum_{k \ge -1} \mu_\q(k+1) \delta_k$ in the sense that we condition the path to hit $-1$ for the first time at time $n+1$. Moreover, such an excursion can be obtained by cyclicly shifting a bridge $S_n$ of this walk (i.e. conditioning the walk to be at $-1$ at time $n+1$, but without the positivity constraint) at the first time it realises its overall minimum, see e.g. Figure~6 in~\cite{Marzouk:Scaling_limits_of_discrete_snakes_with_stable_branching}; this operation is called a discrete Vervaat transform, see e.g. Pitman~\cite[Chapter~6.1]{Pitman:Combinatorial_stochastic_processes} for details.
Our claim holds when $W_n$ is replaced by $S$, in which case we may take $s=0$; indeed, according to Kortchemski~\cite[Proposition 8]{Kortchemski:Sub_exponential_tail_bounds_for_conditioned_stable_Bienayme_Galton_Watson_trees}, it holds that
\[\Pr{S(n s) - \min_{s \le r \le t} S(n r) > u B_{n |t-s|}} \le c_1 \exp(- c_2 u^\delta).\]
Since $n^{-1/\alpha} B_n$ is slowly varying at infinity, the so-called Potter bounds (see e.g.\cite[Lemma~4.2]{Bjornberg-Stefansson:Random_walk_on_random_infinite_looptrees} or~\cite[Equation~9]{Kortchemski:Sub_exponential_tail_bounds_for_conditioned_stable_Bienayme_Galton_Watson_trees}) assert that for every $\varepsilon > 0$, there exists a constant $c$ depending only on $\varepsilon$ such that for every $n$ large enough,
\[\frac{(n |t-s|)^{-1/\alpha} B_{n |t-s|}}{n^{-1/\alpha} B_{n}} \le c \cdot |t-s|^{-\varepsilon},\]
and so
\[\Pr{S(n s) - \min_{s \le r \le t} S(n r) > c \cdot u \cdot |t-s|^{-\varepsilon+1/\alpha} \cdot B_n} \le c_1 \exp(- c_2 u^\delta).\]
One can then transfer this bound to $S_n$; an argument based on the Markov property of $S$ indeed results in an absolute continuity between the first $n/2$ steps of $S$ and of $S_n$, see e.g.~\cite{Kortchemski:Sub_exponential_tail_bounds_for_conditioned_stable_Bienayme_Galton_Watson_trees}, near the end of the proof of Theorem~9 there. Finally, we can transfer this bound from $S_n$ to $W_n$ using the preceding construction from a cyclic shift, see e.g. the end of the proof of Equation~7 in~\cite{Marzouk:Scaling_limits_of_discrete_snakes_with_stable_branching}.

In the case $A=\{0\}$, the construction of $W_n$ from a bridge $S_n$ is discussed by Kortchemski~\cite[Section~6.1]{Kortchemski:Invariance_principles_for_Galton_Watson_trees_conditioned_on_the_number_of_leaves}, and as discussed in Section~8 there, and it extends \emph{mutatis mutandis} to the general case $A$ either finite or co-finite. Here, the bridge $S_n$ is obtained by conditioning the walk $S$ the be at $-1$ after its $n$-th jump in the set $A-1$. Therefore, it suffices again to prove our claim when $W_n$ is replaced by $S_n$ and $|t-s| < 1/2$. Again, we may cut the path of $S_n$ at the time it realises its $(n/2)$-th jump in the set $A-1$ and this path is absolutely continuous with respect to that of the unconditioned walk $S$ cut at the analogous stopping time. This follows from the same argument as alluded above, appealing to the strong Markov property. So finally, we have reduced our claim to showing that it holds when $W_n$ is replaced by $S$, when $s=0$, and $\zeta(T_{A,n})$ is replaced by the time of the $n$-th jump of $S$ in the set $A-1$. This random time, divided by $n$ converges almost surely towards $\mu_q(A)$, see e.g. \cite[Lemma~6.2]{Kortchemski:Invariance_principles_for_Galton_Watson_trees_conditioned_on_the_number_of_leaves} so we may replace it by $(1\pm \gamma) \mu_q(A) n$ with a given $\gamma \in (0,1)$ and conclude from the previous bound on $S$ (it only affects the constants).
\end{proof}

We next turn to the proof of tightness of the label process. Recall that we may replace $B_{\zeta(T_{A,n})}$ by $B_n$. Our argument closely follows the proof of Proposition~7 in~\cite{Marzouk:Scaling_limits_of_random_bipartite_planar_maps_with_a_prescribed_degree_sequence}.

\begin{proof}
[Proof of the tightness in Theorem~\ref{thm:cv_serpents_cartes}]
Fix $q > \frac{2\alpha}{\alpha-1}$ and $\beta \in (1, \frac{q(\alpha-1)}{2\alpha})$. We aim at showing that for every $n$ large enough, for every pair $0 \le s \le t \le 1$, it holds that
\begin{equation}\label{eq:moments_labels}
\Es{|L_n(\zeta(T_{A,n}) s) - L_n(\zeta(T_{A,n}) t)|^q} \le C(q) \cdot B_n^{q/2} \cdot |t-s|^{\beta},
\end{equation}
where here and in all this proof, $C(q)$ stands for some constant, which will vary from one equation to the other, which depends on $q$, $\beta$, and the offspring distribution, but not on $n$ nor $s$ nor $t$. Tightness follows from~\eqref{eq:moments_labels} appealing to the standard Kolmogorov's criterion.

Without loss of generality, we may, and do, restrict to those times $s$ and $t$ such that $|t-s| \le 1/2$ and both $\zeta(T_{A,n}) s$ and $\zeta(T_{A,n}) t$ are integers. Let us then denote by $u$ and $v$ the vertices corresponding to the times $\zeta(T_{A,n}) s$ and $\zeta(T_{A,n}) t$ respectively in lexicographical order, so $L_n(\zeta(T_{A,n}) s) - L_n(\zeta(T_{A,n}) t) = \ell(u)-\ell(v)$. Let $u \wedge v$, be the most recent common ancestor of $u$ and $v$ and further $\hat{u}$ and $\hat{v}$ be the children of $u \wedge v$ which are respectively ancestor of $u$ and $v$. We stress that $u$ and $v$ correspond to deterministic times, whereas $u \wedge v$, $\hat{u}$ and $\hat{v}$ correspond to random times which are measurable with respect to $T_{A,n}$. We write:
\[\ell(u) - \ell(v) = \left(\sum_{w \in \mathopen{\rrbracket} \hat{u}, u \mathclose{\rrbracket}} \ell(w) - \ell(pr(w))\right) + (\ell(\hat{u}) - \ell(\hat{v})) + \left(\sum_{w \in \mathopen{\rrbracket} \hat{v}, v \mathclose{\rrbracket}} \ell(pr(w)) - \ell(w)\right).\]

Recall the notation $1 \le \chi_{\hat{u}} \le \chi_{\hat{v}} \le k_{u \wedge v}$ for the relative position of $\hat{u}$ and $\hat{v}$ among the children of $u \wedge v$. By construction of the labels on $T_{A,n}$, conditional on the tree, the difference $\ell(\hat{u}) - \ell(\hat{v})$ is distributed as $X_{p,i} - X_{p,j}$ with $p = k_{u \wedge v}$, $i = \chi_{\hat{u}}$ and $j = \chi_{\hat{v}}$ and where $X_p$ has the uniform distribution on the set of bridges with no-negative jumps $\mathcal{B}_p^+$. According to Le Gall \& Miermont~\cite[Lemma~1]{Le_Gall-Miermont:Scaling_limits_of_random_planar_maps_with_large_faces}, we thus have
\[\Esc{\left|\ell(\hat{u}) - \ell(\hat{v})\right|^q}{T_{A,n}} \le C(q) \cdot (\chi_{\hat{v}} - \chi_{\hat{u}})^{q/2}.\]
Next, fix $w \in \mathopen{\rrbracket} \hat{u}, u \mathclose{\rrbracket}$, since $\ell(pr(w)) = \ell(pr(w) k_{pr(w)})$, similarly, we have
\[\Esc{|\ell(w) - \ell(pr(w))|^q}{T_{A,n}}
\le C(q) \cdot (k_{pr(w)} - \chi_w)^{q/2},\]
and, for every $w \in \mathopen{\rrbracket} \hat{v}, v \mathclose{\rrbracket}$,
\[\Esc{|\ell(pr(w)) - \ell(w)|^q}{T_{A,n}}
\le C(q) \cdot \chi_w^{q/2}.\]

It was argued in~\cite[Equation~20]{Marzouk:Scaling_limits_of_random_bipartite_planar_maps_with_a_prescribed_degree_sequence}, appealing to the so-called Marcinkiewicz--Zygmund inequality, that if $Y_1, \dots, Y_m$ are independent and centred random variables which admit a finite $q$-th moment, then
\[\Es{\left|\sum_{i=1}^m Y_i\right|^q}
\le C(q) \cdot \left(\sum_{i=1}^m \Es{\left|Y_i\right|^q}^{2/q}\right)^{q/2}.\]
In our context, this reads
\begin{align}\label{eq:tension_branches_gauche_droite}
\Esc{|\ell(u) - \ell(v)|^q}{T_{A,n}}
&\le C(q) \cdot 
\left(
\sum_{w \in \mathopen{\rrbracket} \hat{u}, u \mathclose{\rrbracket}} (k_{pr(w)} - \chi_w)
+ (\chi_{\hat{v}} - \chi_{\hat{u}})
+ \sum_{w \in \mathopen{\rrbracket} \hat{v}, v \mathclose{\rrbracket}} \chi_w
\right)^{q/2}\nonumber
\\
&\le C(q) \cdot 
\left(
\left(\sum_{w \in \mathopen{\rrbracket} \hat{u}, u \mathclose{\rrbracket}} (k_{pr(w)} - \chi_w) + (\chi_{\hat{v}} - \chi_{\hat{u}})\right)^{q/2}
+ \left(\sum_{w \in \mathopen{\rrbracket} \hat{v}, v \mathclose{\rrbracket}} \chi_w\right)^{q/2}
\right).
\end{align}

Let us first consider the first term in~\eqref{eq:tension_branches_gauche_droite}. Appealing to Lemma~\ref{lem:codage_marche_Luka}, we have
\[\chi_{\hat{v}} - \chi_{\hat{u}}
= W_n(\hat{u}) - W_n(\hat{v}),\]
and similarly, for every $w \in \mathopen{\rrbracket} \hat{u}, u \mathclose{\rrbracket}$,
\[k_{pr(w)} - \chi_w
= W_n(w) - W_n(pr(w) k_{pr(w)})
= W_n(w k_w) - W_n(pr(w) k_{pr(w)}),\]
so
\[\sum_{w \in \mathopen{\rrbracket} \hat{u}, u \mathclose{\rrbracket}} (k_{pr(w)} - \chi_w) + (\chi_{\hat{v}} - \chi_{\hat{u}})
= W_n(u) - W_n(\hat{v})
= W_n(u) - \inf_{[u, v]} W_n.\]
Then Lemma~\ref{lem:moments_marche_Luka} applied with $\theta = 2\beta/q < 1/\alpha$ yields
\[\Es{\left(\sum_{w \in \mathopen{\rrbracket} \hat{u}, u \mathclose{\rrbracket}} (k_{pr(w)} - \chi_w) + (\chi_{\hat{v}} - \chi_{\hat{u}})\right)^{q/2}}
\le C(q) \cdot B_n^{q/2} \cdot |t-s|^\beta.\]

We next focus on the second term in~\eqref{eq:tension_branches_gauche_droite}. 
We would like to proceed symmetrically but there is a technical issue: on the branch $\mathopen{\rrbracket} \hat{u}, u \mathclose{\rrbracket}$, we relied on the fact that $\ell(wk_w) = \ell(w)$ in order to only count the number of vertices branching off of this path \emph{strictly} to the right, but this is not the case on $\mathopen{\rrbracket} \hat{v}, v \mathclose{\rrbracket}$: we do not have $\ell(w1)=\ell(w)$ in general so we must also count the vertices \emph{on} this path.
Let $T_{A,n}^-$ be the `mirror image' of $T_{A,n}$, i.e. the tree obtained from $T_{A,n}$ by flipping the order of the children of every vertex; let us write $w^- \in T_{A,n}^-$ for the mirror image of a vertex $w \in T_{A,n}$; make the following observations:
\begin{enumerate}
\item $T_{A,n}^-$ has the same law as $T_{A,n}$, so in particular, their {\L}ukasiewicz paths have the same law;
\item for every $w \in \mathopen{\rrbracket} \hat{v}, v \mathclose{\rrbracket}$, the quantity $\chi_w-1$ in $T_{A,n}$ corresponds to the quantity $k_{pr(w^-)} - \chi_{w^-}$ in $T_{A,n}^-$;
\item the lexicographical distance between the last descendant in $T_{A,n}^-$ of respectively $\hat{v}^-$ and $v^-$ is smaller than the lexicographical distance between $\hat{v}$ and $v$ in $T_{A,n}$ (the elements of $ \mathopen{\rrbracket} \hat{v}, v \mathclose{\rrbracket} = \mathopen{\rrbracket} \hat{v}^-, v^- \mathclose{\rrbracket}$ are missing). 
\end{enumerate}
With theses observations, the previous argument used to control the branch $\mathopen{\rrbracket} \hat{u}, u \mathclose{\rrbracket}$ shows that
\[\Es{\left(\sum_{w \in \mathopen{\rrbracket} \hat{v}, v \mathclose{\rrbracket}} (\chi_w-1)\right)^{q/2}}
\le C(q) \cdot B_n^{q/2} \cdot |t-s|^\beta.\]
Finally, as proved recently (for the conditioning $A = \Z_+$ but the general case follows similarly) in~\cite{Marzouk:Scaling_limits_of_discrete_snakes_with_stable_branching}: for every $\gamma < (\alpha-1)/\alpha$,
\begin{equation}\label{eq:hauteur_Holder}
\Es{\# \mathopen{\rrbracket} \hat{v}, v \mathclose{\rrbracket}^{q/2}} \le C(q) \cdot \left(\frac{B_{\zeta(T_{A,n})}}{\zeta(T_{A,n})}\right)^{q/2} \cdot |t-s|^{\gamma q/2},
\end{equation}
which is smaller than the bound we are looking for; indeed, since we assume that both $\zeta(T_{A,n}) s$ and $\zeta(T_{A,n}) t$ are integers, then $\zeta(T_{A,n})^{-q/2} \le |t-s|^{q/2} \le 1$ and $\gamma$ can be chosen close enough to $(\alpha-1)/\alpha$ to ensure that $|t-s|^{(\gamma+1) q/2} \le |t-s|^\beta$.
\end{proof}

Let us mention that we have hidden the technical difficulties in~\eqref{eq:hauteur_Holder}. Nevertheless, there is a different argument which does not necessitate any control on the length of the branches. Indeed the bound~\eqref{eq:hauteur_Holder} answers in this context of size-conditioned Bienaymé--Galton--Watson trees Remark~3 in~\cite{Marzouk:Scaling_limits_of_random_bipartite_planar_maps_with_a_prescribed_degree_sequence} on trees with a prescribed degree sequence. We could have argued instead as in the proof of Proposition~7 there that, if $\chi_w \ge 2$, then $\chi_w \le 2 (\chi_w-1)$, so in order to control the moments of $\sum_{w \in \mathopen{\rrbracket} \hat{v}, v \mathclose{\rrbracket}} \chi_w$, it suffices to bound those of $\#\{w \in \mathopen{\rrbracket} \hat{v}, v \mathclose{\rrbracket} : \chi_w = 1\}$. But according to~\cite[Lemma~2]{Marzouk:Scaling_limits_of_discrete_snakes_with_stable_branching} (which recasts~\cite[Corollary~3]{Marzouk:Scaling_limits_of_random_bipartite_planar_maps_with_a_prescribed_degree_sequence} in the context of size-conditioned Bienaymé--Galton--Watson trees), with high probability, uniformly for all pair of vertices $\hat{v},v$ such that $\hat{v}$ is an ancestor of $v$,%
\footnote{And the path $\llbracket \hat{v}, v\rrbracket$ has length at least of order $\ln n$, but shorter paths do not cause any issue.}
there is a proportion at most $1-\mu_\q(0)/2 < 1$ of individuals $w \in \mathopen{\rrbracket} \hat{v}, v \mathclose{\rrbracket}$ such that $\chi_w = 1$. Then the bound~\eqref{eq:moments_labels} holds under the conditional expectation with respect to this event, so tightness of the label process holds conditional on this event, and so also unconditionally.

\subsection{Finite dimensional marginals in the non-Gaussian case}

In this subsection, we assume that $\alpha < 2$, and prove the following result which, together with the tightness obtained in the preceding subsection, concludes the proof of Theorem~\ref{thm:cv_serpents_cartes} in this case.

\begin{prop}
For every $k \ge 1$ and every $0 \le t_1 < \dots < t_k \le 1$, it holds that
\[\left(B_{\zeta(T_{A,n})}^{-1/2} L_n(\lfloor \zeta(T_{A,n}) t_i\rfloor)\right)_{1 \le i \le k}
\cvloi (\Lab_{t_i})_{1 \le i \le k},\]
jointly with~\eqref{eq:Duquesne_Kortchemski}.
\end{prop}

Our argument follows closely that of Le Gall \& Miermont~\cite[proof of Proposition~7]{Le_Gall-Miermont:Scaling_limits_of_random_planar_maps_with_large_faces} who considered the two-type tree associated with the maps via the Bouttier--Di Francesco--Guitter bijection, whereas we use the Janson--Stef{\'a}nsson bijection which eliminates several technicalities. The argument relies of the convergence of the {\L}ukasiewicz path in~\eqref{eq:Duquesne_Kortchemski} which we assume for the rest of this subsection to hold almost surely, appealing to Skorokhod's representation Theorem. To ease the exposition, we start with the one-dimensional marginals.

\begin{proof}[Proof in the case $k=1$]
Fix $t \in [0,1]$ and recall the notation $\mathscr{I}_{s, t} = \inf_{r \in [s,t]} \Xexc_r$ for every $s \in [0,t]$. Let $(s_i)_{i \ge 1}$ be those times $s \in [0,t]$ such that
\[\Xexc_{s-} < \mathscr{I}_{s, t},\]
which are ranked in decreasing order of the values of the jumps of $\Xexc$: $\Delta \Xexc_{s_1} > \Delta \Xexc_{s_2} > \dots$. Similarly, let $\kappa_n$ be the number of integers $k \in \{0, \dots, \lfloor \zeta(T_{A,n}) t\rfloor-1\}$ such that
\[W_n(k) = \min_{r \in [k, \lfloor \zeta(T_{A,n}) t \rfloor]} W_n(r),\]
and let us denote by $a_{n,1}, \dots, a_{n, \kappa_n}$ these integers, ranked so that
\[W_n(a_{n,1}+1)-W_n(a_{n,1}) \ge \dots \ge W_n(a_{n,\kappa_n}+1)-W_n(a_{n,\kappa_n}).\]
It follows from~\eqref{eq:Duquesne_Kortchemski} that almost surely, for every $i \ge 1$, we have
\begin{equation}\label{eq:temps_sauts_ancetres}
\begin{aligned}
\frac{1}{\zeta(T_{A,n})} a_{n,i} &\cv s_i,
\\
\frac{1}{B_{\zeta(T_{A,n})}} \left(W_n(a_{n,i}+1)-W_n(a_{n,i})\right) &\cv \Delta \Xexc_{s_i},
\\
\frac{1}{B_{\zeta(T_{A,n})}} \left(\min_{k \in [a_{n,i}+1, \lfloor \zeta(T_{A,n}) t \rfloor]} W_n(k)-W_n(a_{n,i})\right) &\cv \mathscr{I}_{s_i, t} - \Xexc_{s_i-},
\end{aligned}
\end{equation}

Let $u_0, u_1, \dots, u_{\zeta(T_{A,n})}$ be the vertices of $T_{A,n}$ listed in lexicographical order. 
Observe that the $a_{n,i}$'s are exactly the indices of the strict ancestors of $u_{\lfloor \zeta(T_{A,n}) t \rfloor}$. We may then write
\[L_n(\lfloor \zeta(T_{A,n}) t \rfloor) = \ell(u_{\lfloor \zeta(T_{A,n}) t \rfloor})
= \sum_{i=1}^{\kappa_n} (\ell(u_{\psi(a_{n,i})}) - \ell(u_{a_{n,i}})),\]
where $u_{\psi(a_{n,i})}$ is the only child of $u_{a_{n,i}}$ which is an ancestor of $u_{\lfloor \zeta(T_{A,n}) t \rfloor}$. We claim that only the first values of $i$ matters. Indeed, by classical results on fluctuation theory, it is well known that
\[\Xexc_t = \sum_{i \ge 1} (\mathscr{I}_{s_i, t} - \Xexc_{s_i-}),\]
whence, for every $\varepsilon > 0$, there exists an integer $N \ge 1$ such that with probability at least $1-\varepsilon$, it holds that
\[\Xexc_t - \sum_{i \le N} (\mathscr{I}_{s_i, t} - \Xexc_{s_i-}) \le \varepsilon/2.\]
Then~\eqref{eq:temps_sauts_ancetres} and~\eqref{eq:Duquesne_Kortchemski} imply that for every $n$ sufficiently large, with probability at least $1-2\varepsilon$, it holds that
\[\frac{1}{B_{\zeta(T_{A,n})}} \left(W_n(\lfloor\zeta(T_{A,n}) t\rfloor) - \sum_{i=1}^{N \wedge \kappa_n} \min_{k \in [a_{n,i}+1, \lfloor \zeta(T_{A,n}) t \rfloor]} W_n(k) - W_n(a_{n,i})\right) < \varepsilon.\]
Observe that the left-hand side equals
\[\frac{1}{B_{\zeta(T_{A,n})}} \sum_{i=N+1}^{\kappa_n} \min_{k \in [a_{n,i}+1, \lfloor \zeta(T_{A,n}) t \rfloor]} W_n(k) - W_n(a_{n,i}),\]
which is therefore arbitrarily small when fixing $N$ large enough. Now recall that, conditional on $T_{A,n}$, the label increments $\ell(u_{\psi(a_{n,i})}) - \ell(u_{a_{n,i}})$ for $1 \le i \le \kappa_n$ are independent and distributed as $X_{k_i, \chi_i}$ where $(X_{k_i, 1}, \dots, X_{k_i, k_i})$ is a uniform random bridge in $\mathcal{B}^+_{k_i}$ defined in~\eqref{eq:pont_sans_saut_negatif}, where $k_i = W_n(a_{n, i}+1) - W_n(a_{n, i}) + 1$ is the number of children of $a_{n,i}$, and where $\chi_i = W_n(a_{n, i}+1) - W_n(\psi(a_{n, i})) + 1$ is the position of $\psi(a_{n,i})$ amongst its siblings; note that $W_n(\psi(a_{n, i})) = \min_{j \in [a_{n, i}+1, \lfloor \zeta(T_{A,n}) t\rfloor]} W_n(j)$. As in the preceding section, according to Le Gall \& Miermont~\cite[Equation~17]{Le_Gall-Miermont:Scaling_limits_of_random_planar_maps_with_large_faces}, there exists some universal constant $K > 0$ such that
\[\Esc{|\ell(u_{\psi(a_{n,i})}) - \ell(u_{a_{n,i}})|^2}{T_{A,n}}
= K\frac{\chi_i (k_i-\chi_i)}{k_i}
\le K \left(\min_{j \in [a_{n, i}+1, \lfloor \zeta(T_{A,n}) t\rfloor]} W_n(j) - W_n(a_{n, i})\right).\]
Since, conditional on $T_{A_n}$, these increments are centred and independent, we conclude that
\begin{multline*}
\Esc{\left|B_{\zeta(T_{A,n})}^{-1/2} \sum_{i=N+1}^{\kappa_n} (\ell(u_{\psi(a_{n,i})}) - \ell(u_{a_{n,i}}))\right|^2}{T_{A,n}}
\\
\le K \cdot B_{\zeta(T_{A,n})}^{-1} \sum_{i=N+1}^{\kappa_n} \min_{j \in [a_{n, i}+1, \lfloor \zeta(T_{A,n}) t\rfloor]} W_n(j) - W_n(a_{n, i}),
\end{multline*}
which, on a set of probability at least $1-2\varepsilon$ for every $n$ large enough, is bounded by $K \varepsilon$ according to the preceding discussion.

We next focus on $a_{n,1}, \dots, a_{n,N}$. Conditional on $\Xexc$, let $(\gamma_i)_{1 \le i \le N}$ be independent Brownian bridges of length $\Delta \Xexc_{s_i}$ respectively; it follows from~\eqref{eq:temps_sauts_ancetres} and~\eqref{eq:Duquesne_Kortchemski} together with Donsker's invariance principle for random bridges~\eqref{eq:Donsker_ponts} that
\[B_{\zeta(T_{A,n})}^{-1/2} \sum_{i=1}^N (\ell(u_{\psi(a_{n,i})}) - \ell(u_{a_{n,i}}))
\cvloi \sqrt{2} \sum_{i=1}^N \gamma_i(\Xexc_{s_i} - \mathscr{I}_{s_i, t}),\]
and the right-hand side converges further towards $\Lab_t$ as $N \to \infty$.
\end{proof}

We next briefly sketch the argument for the multi-dimensional marginals.

\begin{proof}[Proof in the case $k \ge 2$]
To ease the notation, we only treat the case $k = 2$, but the arguments are valid in the more general case. Let us fix $0 < s < t$; let us denote by $0 = a_{n,0}' < \dots < a_{n,\kappa_n'}'$ the indices of the strict ancestors of $u_{\lfloor \zeta(T_{A,n}) s\rfloor}$, and let similarly $0 = a_{n,0}'' < \dots < a_{n,\kappa_n''}''$ be the indices corresponding to the ancestors of $u_{\lfloor \zeta(T_{A,n}) t\rfloor}$. Let $j(n) \in \{0, \dots, \lfloor \zeta(T_{A,n}) s\rfloor\}$ be the index such that $u_{j(n)}$ is the last common ancestor of $u_{\lfloor \zeta(T_{A,n}) s\rfloor}$ and $u_{\lfloor \zeta(T_{A,n}) t\rfloor}$; we implicitly assume that $j(n) < \lfloor \zeta(T_{A,n}) s\rfloor$ but this case is treated similarly. Let $i(n) \in \{0, \dots \kappa_n' \wedge \kappa_n''\}$ be the index such that $j(n) = a_{n,i(n)}' = a_{n,i(n)}''$. Note that $j(n)$ is the unique time such that
\[W_n(j(n))
\le \min_{k \in [\lfloor \zeta(T_{A,n}) s\rfloor, \lfloor \zeta(T_{A,n}) t\rfloor]} W_n(k)
< \min_{k \in [j(n)+1, \lfloor \zeta(T_{A,n}) s\rfloor]} W_n(k).\]
By analogy with the discrete setting, we interpret the times $r \in [0,s]$ such that $\Xexc_{r-} < \mathscr{I}_{r, s}$ as the times of visit of the ancestors of the vertex visited at time $s$, and similarly for $t$. Introduce then the unique time $r_0 \in [0,s]$ such that
\[\Xexc_{r_0-} < \mathscr{I}_{s, t} < \mathscr{I}_{r_0, s},\]
which intuitively corresponds to the time of visit of the last common ancestor of the vertices visited at time $s$ and $t$, and indeed, from~\eqref{eq:Duquesne_Kortchemski}, it is the almost sure limit of $\zeta(T_{A,n})^{-1} j(n)$.
Let us consider the label increments at the branch-point: conditional on $\Xexc$, let $\gamma$ be a Brownian bridge of length $\Delta \Xexc_{r_0}$, then similar arguments as in the one-dimensional case show that the pair
\[B_{\zeta(T_{A,n})}^{-1/2} \left(L_n(a_{n,i(n)+1}') - L_n(j(n)), L_n(a_{n,i(n)+1}'') - L_n(j(n))\right)\]
converges in distribution as $n\to\infty$ towards
\[\sqrt{2} \cdot (\gamma(\Xexc_{r_0} - \mathscr{I}_{r_0, s}), \gamma(\Xexc_{r_0} - \mathscr{I}_{r_0, t})).\]
If one removes the branch-point from the subtree of $T_{A,n}$ spanned by its root and the vertices $u_{\lfloor \zeta(T_{A,n}) s\rfloor}$ and $u_{\lfloor \zeta(T_{A,n}) t\rfloor}$, then one gets three branches and the label increments between their root and their leaf are independent; we may apply the arguments of the previous proof to prove that these three increments, divided by $B_{\zeta(T_{A,n})}^{1/2}$ converge in distribution towards the `label increments' given by $\Lab$. Details are left to the reader, we refer to the end of the proof of Proposition~7 in~\cite{Le_Gall-Miermont:Scaling_limits_of_random_planar_maps_with_large_faces}.
\end{proof}

\subsection{Finite dimensional marginals in the Gaussian case}

We now focus on the Gaussian regime $\alpha=2$. As opposed to the other regimes, we consider \emph{random} marginals. Precisely, we prove the following result.

\begin{prop}\label{prop:marginal_etiquettes_cas_gaussien}
For every $k \ge 1$, sample $U_1, \dots, U_k$ i.i.d. uniform random variables in $[0,1]$ independently of the labelled trees, then the convergence
\[\left(B_{\zeta(T_{A,n})}^{1/2} L_n(\lfloor \zeta(T_{A,n}) U_i\rfloor)\right)_{1 \le i \le k}
\cvloi (\Lab_{U_i})_{1 \le i \le k},\]
holds jointly with~\eqref{eq:Duquesne_Kortchemski}, where the process $\Lab$ is independent of $U_1, \dots, U_k$.
\end{prop}

Since we know that the sequence of continuous processes $B_{\zeta(T_{A,n})}^{-1/2} L_n(\zeta(T_{A,n}) \cdot)$ is tight, this suffices to characterise the subsequential limits as $\Lab$. Indeed, given any finite collection of fixed times in $[0,1]$, one can approximate them by sampling sufficiently many i.i.d. uniform random variables in $[0,1]$; then the equicontinuity given by the tightness shows that the images of $L_n$ at these random times approximate well the values at the deterministic times, and the same hods for the uniformly continuous limit $\Lab$, see e.g. Addario-Berry \& Albenque~\cite[proof of Proposition 6.1]{Addario_Berry-Albenque:The_scaling_limit_of_random_simple_triangulations_and_random_simple_triangulations} for a detailed argument.

As previously, we first treat the one-dimensional case.

\begin{proof}[Proof in the case $k=1$]
The approach was described earlier in this section. 
Sample $U$ uniformly at random in $[0,1]$ independently of the rest and note that the vertex $u_n$ visited at the time $\lceil \zeta(T_{A,n}) U \rceil$ in lexicographical order has the uniform distribution in $T_n$;%
\footnote{Precisely $u_n$ has the uniform distribution in $T_n \setminus \{\varnothing\}$, but we omit this detail for the sake of clarity.} 
let us write
\[\left(\frac{1}{B_{\zeta(T_{A,n})}}\right)^{1/2} \ell(u_n)
= \left(\frac{B_{\zeta(T_{A,n})}}{\zeta(T_{A,n})} |u_n|\right)^{1/2} \cdot \left(\frac{\zeta(T_{A,n})}{B_{\zeta(T_{A,n})}^2 |u_n|}\right)^{1/2} \ell(u_n).\]
It follows from~\eqref{eq:Duquesne_Kortchemski} that the first term on the right converges in distribution towards $\Hexc_U$, it is therefore equivalent to show that, jointly with~\eqref{eq:Duquesne_Kortchemski}, we have the convergence in distribution
\begin{equation}\label{eq:cv_label_unif}
\left(\frac{3 \zeta(T_{A,n})}{2 B_{\zeta(T_{A,n})}^2 |u_n|}\right)^{1/2} \ell(u_n)
\cvloi G
\end{equation}
where $G$ has the standard Gaussian distribution. Recall that according to Remark~\ref{rem:nombre_aretes_arbre}, we may, and do, replace $\zeta(T_{A,n})$ by $\mu_\q(A)^{-1} n$ and $B_{\zeta(T_{A,n})}$ by $\mu_\q(A)^{-1/2} B_n$.

For every $k \ge j \ge 1$, let us denote by $A_{k,j}(u_n)$ the number of strict ancestors of $u_n$ with $k$ children, among which the $j$-th one is again an ancestor of $u_n$:
\[A_{k,j}(u_n) =\# \left\{v \in \llbracket\varnothing, u_n\llbracket : k_v = k \text{ and } vj \in \rrbracket\varnothing, u_n\rrbracket\right\}.\]
The idea is to decompose $\ell(u_n)$ as the sum of the label increments between two consecutive ancestors $w$ and $pr(w)$; conditionally on $T_n$, these random variables are independent and, whenever $k_{pr(w)}=k$ and $w = pr(w)j$, the label increment has the law of $j$-th marginal of a uniform random bridge in $\mathcal{B}_k^+$, which is centred and has variance, say, $\sigma_{k,j}^2$. This variance is known explicitly, see e.g.~\cite[page 1664%
\footnote{Note that they consider uniform random bridges in $\mathcal{B}_{k+1}^+$!}]{Marckert-Miermont:Invariance_principles_for_random_bipartite_planar_maps}: we have
\begin{equation}\label{eq:variance_accroissements_labels_cas_gaussien}
\sigma_{k,j}^2 = \frac{2j(k-j)}{k+1},
\qquad\text{so}\qquad
\sum_{j=1}^k \sigma_{k,j}^2 = \frac{k(k-1)}{3}.
\end{equation}
Let $\Delta(T_{A,n})$ denote the largest offspring of a vertex of $T_{A,n}$. As in the classical proof of the central limit theorem, we may write for every $z \in \R$,
\begin{align*}
&\Esc{\exp\left(\i z \left(\frac{3 n}{2 B_n^2 |u_n|}\right)^{1/2} \ell(u_n)\right)}{T_n, u_n}
\\
&= \prod_{k=1}^{\Delta(T_{A,n})} \prod_{j=1}^k \left(1 - \frac{z^2}{2} \frac{3 n \sigma_{k,j}^2}{2 B_n^2 |u_n|} + o\left(\left(\frac{n \sigma_{k,j}^2}{B_n^2 |u_n|}\right)^2\right)\right)^{A_{k,j}(u_n)}
\\
&= \exp\left(- \frac{z^2}{2} \sum_{k=1}^{\Delta(T_{A,n})} \sum_{j=1}^k A_{k,j}(u_n) \left(\frac{3 n \sigma_{k,j}^2}{2 B_n^2 |u_n|}  + O\left(\left(\frac{n \sigma_{k,j}^2}{B_n^2 |u_n|}\right)^2\right)\right)\right).
\end{align*}
We claim that
\begin{equation}\label{eq:moyenne_accroissements_etiquettes_cas_gaussien}
\sum_{k=1}^{\Delta(T_{A,n})} \sum_{j=1}^k \frac{3 n \sigma_{k,j}^2}{2 B_n^2 |u_n|} A_{k,j}(u_n)
\cvproba 1,
\quad\text{and}\quad
\sum_{k=1}^{\Delta(T_{A,n})} \sum_{j=1}^k \left(\frac{n \sigma_{k,j}^2}{B_n^2 |u_n|}\right)^2 A_{k,j}(u_n)
\cvproba 0.
\end{equation}
Then an application of Lebesgue's Theorem yields our claim.

We shall restrict ourselves to a `good event' that we now introduce. For a vertex $u \in T_{A,n}$, let $LR(u)$ denote the number of vertices branching off of the path $\llbracket \varnothing, u \llbracket$ i.e. whose parent belongs to this ancestral line, and which themselves do not; formally,
\[LR(u) = \#\left\{v \in T_{A,n}\setminus \llbracket \varnothing, u \rrbracket : pr(v) \in \llbracket \varnothing, u \llbracket \right\}.\]
For a tree $T$, a vertex $u \in T$ and three (small) parameters $\eta, \gamma, \kappa > 0$, let us consider the event
\[E_{n, \eta, \gamma, \kappa}(T, u) = 
\left\{\Delta(T) \le \eta B_n\right\}
\cap \left\{\gamma \le \frac{B_n}{n} |u| \le \gamma^{-1}\right\}
\cap \left\{LR(u) \ge \kappa B_n\right\}.\]
If $u$ is the $k$-th vertex of $T$ in lexicographical order and $W$ and $H$ and respectively the {\L}ukasiewicz path and the height process of $T$, then $\Delta(T)$ is the largest jump plus one of $W$ and $|u|$ equals $H(k)$. Finally, if $u_n$ is a uniform random vertex of $T_{A,n}$, then $LR(u_n)$ has the law of the sum of two independent copies of $W(U_n)$ where $U_n$ is a uniform random integer in $\{0, \dots, \zeta(T_{A,n})\}$ independent of $T_{A,n}$; indeed, the number of vertices branching off to the right of the path $\llbracket \varnothing, u_n \llbracket$ is exactly $W(U_n)$, and then by symmetry, the number of vertices branching off to the left is the value of the `mirror {\L}ukasiewicz path' at the corresponding time; we then conclude by the invariance of the law of the tree by this `mirror' operation. It thus follows from~\eqref{eq:Duquesne_Kortchemski} that for any $\eta > 0$,
\[\lim_{\gamma, \kappa \downarrow 0} \liminf_{n \to \infty} \Pr{E_{n, \eta, \gamma, \kappa}(T_{A,n}, u_n)} = 1.\]
Now fix $\varepsilon, \delta > 0$ and choose $\gamma, \kappa$ small enough so that for any $\eta > 0$, the probability of $E_{n, \eta, \gamma, \kappa}(T_{A,n}, u_n)$ is at least $1-\delta$ for every $n$ large enough. We shall tune $\eta$ in such a way that for every $n$ large enough,
\[\Pr{\left\{\left|\sum_{k=1}^{\Delta(T_{A,n})} \sum_{j=1}^k \frac{3 n \sigma_{k,j}^2}{2 B_n^2 |u_n|} A_{k,j}(u_n) - 1\right| > \varepsilon\right\} \cap E_{n, \eta, \gamma, \kappa}(T_{A,n}, u_n)} \le \delta.\]

We appeal to a \emph{spinal decomposition} due to Duquesne~\cite[Equation~24]{Duquesne:An_elementary_proof_of_Hawkes_s_conjecture_on_Galton_Watson_trees} which results in an absolute continuity relation between the tree $T_{A,n}$ and the tree $T_\infty$ `conditioned to survive', which is the infinite tree which arises as the \emph{local limit} of $T_{A,n}$. It was introduced by Kesten~\cite{Kesten:Subdiffusive_behavior_of_random_walk_on_a_random_cluster} and the most general results on such convergences are due to Abraham \& Delmas~\cite{Abraham-Delmas:Local_limits_of_conditioned_Galton_Watson_trees_the_infinite_spine_case}. The tree $T_\infty$ contains a unique infinite simple path called the \emph{spine}, starting from the root, and the vertices which belong to this spine reproduce according to the \emph{size-biased} law $\sum_{k \ge 1} k \mu_\q(k) \delta_k$, whereas the other vertices reproduce according to $\mu_\q$, and all the vertices reproduce independently. For a tree $\tau$ and a vertex $v \in \tau$, let $\theta_v(\tau)$ be the subtree consisting of $v$ and all its progeny, and let $\mathsf{Cut}_v(\tau) = \{v\} \cup (\tau \setminus \theta_v(\tau))$ be its complement (note that $v$ belongs to both parts). Then for every non-negative measurable functions $G_1, G_2$, for every $h \ge 0$, we have
\begin{equation}\label{eq:epine}
\Es{\sum_{\substack{v \in T \\ |v| = h}} G_1(\mathsf{Cut}_v(T), v) \cdot G_2(\theta_v(T))}
= \Es{G_1(\mathsf{Cut}_{v_h^\ast}(T_\infty), v_h^\ast) \cdot G_2(T')},
\end{equation}
where $v_h^\ast$ is the only vertex on the spine of $T_\infty$ at height $h$ and where $T'$ is independent of $T_\infty$ and distributed as a non-conditioned Bienaymé--Galton--Watson tree with offspring distribution $\mu_\q$.

Observe that the $A_{k,j}(u_n)$ are measurable with respect to $(u_n, \mathsf{Cut}_{u_n}(T_{A,n}))$ and so are $|u_n|$ and $LR(u_n)$; finally, let us replace the event $E_{n, \eta, \gamma, \kappa}(T_{A,n}, u_n)$ by $E_{n, \eta, \gamma, \kappa}(\mathsf{Cut}_{u_n}(T_{A,n}), u_n)$ whose probability is greater. 
Let $T$ be a a non-conditioned Bienaymé--Galton--Watson tree with offspring distribution $\mu_\q$ and let $\zeta_A(T)$ be its number of vertices with offspring in $A$. Let us write
\begin{multline*}
\Pr{\left\{\left|\sum_{k=1}^{\Delta(\mathsf{Cut}_{u_n}(T_{A,n}))} \sum_{j=1}^k \frac{3 n \sigma_{k,j}^2}{2 B_n^2 |u_n|} A_{k,j}(u_n) - 1\right| > \varepsilon\right\} \cap E_{n, \eta, \gamma, \kappa}(\mathsf{Cut}_{u_n}(T_{A,n}), u_n)}
\\
= \frac{1}{\P(\zeta_A(T) = n)} \Pr{\left\{\left|\sum_{k=1}^{\eta B_n} \sum_{j=1}^k \frac{3 n \sigma_{k,j}^2}{2 B_n^2 |u_n|} A_{k,j}(u_n) - 1\right| > \varepsilon\right\} \cap \{\zeta_A(T) = n\} \cap E_{n, \eta, \gamma, \kappa}(\mathsf{Cut}_{u_n}(T), u_n)}.
\end{multline*}
According to~\cite[Theorem~8.1]{Kortchemski:Invariance_principles_for_Galton_Watson_trees_conditioned_on_the_number_of_leaves}, the quantity $(n B_n)^{-1} \P(\zeta_A(T) = n)$ converges to some positive and finite limit. Then by conditioning on the value of $u_n$, we have
\begin{multline*}
\Pr{\left\{\left|\sum_{k=1}^{\eta B_n} \sum_{j=1}^k \frac{3 n \sigma_{k,j}^2}{2 B_n^2 |u_n|} A_{k,j}(u_n) - 1\right| > \varepsilon\right\} \cap \{\zeta_A(T) = n\} \cap E_{n, \eta, \gamma, \kappa}(\mathsf{Cut}_{u_n}(T), u_n)}
\\
= \frac{1}{n} \sum_{h = \gamma n/B_n}^{\gamma^{-1} n/B_n} 
\Es{\sum_{\substack{v \in T \\ |v| = h}} \ind{\{|\sum_{k=1}^{\eta B_n} \sum_{j=1}^k \frac{3 n \sigma_{k,j}^2}{2 B_n^2 h} A_{k,j}(v) - 1| > \varepsilon\} \cap \{\zeta_A(T) = n\} \cap E_{n, \eta, \gamma, \kappa}(\mathsf{Cut}_v(T), v)}}.
\end{multline*}
This is almost in the form of~\eqref{eq:epine}, we just need to express the quantity $\zeta_A(T)$ in terms of $\mathsf{Cut}_v(T)$, $v$, and $\theta_v(T)$. Let $\lambda(\mathsf{Cut}_v(T))$ be the number of leaves of the tree $\mathsf{Cut}_v(T)$; one of them is $v$ who gives birth to a progeny coded by $\theta_v(T)$, whereas the other leaves give birth to progenies coded by independent non-conditioned Bienaymé--Galton--Watson trees with offspring distribution $\mu_\q$. By splitting the contribution to $\zeta_A(T)$ of these trees and of $\mathsf{Cut}_v(T)$, the spinal decomposition~\eqref{eq:epine} reads
\begin{multline*}
\Es{\sum_{\substack{v \in T \\ |v| = h}} \ind{\{|\sum_{k=1}^{\eta B_n} \sum_{j=1}^k \frac{3 n \sigma_{k,j}^2}{2 B_n^2 h} A_{k,j}(v) - 1| > \varepsilon\} \cap \{\zeta_A(T) = n\} \cap E_{n, \eta, \gamma, \kappa}(\mathsf{Cut}_v(T), v)}}
\\
= \P\left(\left\{\left|\sum_{k=1}^{\eta B_n} \sum_{j=1}^k \frac{3 n \sigma_{k,j}^2}{2 B_n^2 h} A_{k,j}(v^\ast_h) - 1\right| > \varepsilon\right\}
\cap E_{n, \eta, \gamma, \kappa}(\mathsf{Cut}_{v^\ast_h}(T_\infty), v^\ast_h)\right.
\\
\left.\cap \{\zeta_A(\mathsf{Cut}_{v^\ast_h}(T_\infty)) + \zeta_A(F_{\lambda(\mathsf{Cut}_{v^\ast_h}(T_\infty))}) = n\}\right),
\end{multline*}
where for every $N \ge 1$, we let $F_N$ denote a forest of i.i.d. non-conditioned Bienaymé--Galton--Watson trees with offspring distribution $\mu_\q$, which is independent of the rest. In the case $A = \Z_+$ when we condition by the total size, it is well-known that $\zeta_{\Z_+}(F_1) = |T|$ belongs to the domain of attraction of a stable law with index $1/2$, and an application of the local limit theorem shows that there exists a constant $C > 0$ such that for every $N, m \ge 1$,
\[\Pr{\zeta_{\Z_+}(F_n) = m} \le \frac{C}{B_N'},\]
where $(B_N')_{N \ge 1}$ is an increasing sequence such that $(N^{-2} B_N')_{N \ge 1}$ is slowly varying and furthermore
\[B_{B_n}' \enskip\mathop{\sim}^{}_{n \to \infty}\enskip n.\]
We refer to the discussion leading to Equation~28 in~\cite{Kortchemski:Sub_exponential_tail_bounds_for_conditioned_stable_Bienayme_Galton_Watson_trees} for more details. This extends to the general case where either $A$ or $\Z_+ \setminus A$ is finite by replacing Equation~26 in~\cite{Kortchemski:Sub_exponential_tail_bounds_for_conditioned_stable_Bienayme_Galton_Watson_trees} by~\cite[Theorem~8.1]{Kortchemski:Invariance_principles_for_Galton_Watson_trees_conditioned_on_the_number_of_leaves}, so $\zeta_A(F_1)$ also belongs to the domain of attraction of a stable law with index $1/2$; this only modifies the preceding constant $C$. In our case, we shall take $N = LR(v_h^\ast)+1$ which, on the event $E_{n, \eta, \gamma, \kappa}(\mathsf{Cut}_{v^\ast_h}(T_\infty), v^\ast_h)$ is at least $\kappa B_n$, so $(B_N')^{-1}$ is at most $B_{\kappa B_n}^{-1} \sim \kappa^{-2} n^{-1}$ as $n \to \infty$.

Let us put together the previous arguments: we have shown that there exists a constant $K$ such that
\begin{align*}
&\Pr{\left\{\left|\sum_{k=1}^{\Delta(T_{A,n})} \sum_{j=1}^k \frac{3 n \sigma_{k,j}^2}{2 B_n^2 |u_n|} A_{k,j}(u_n) - 1\right| > \varepsilon\right\} \cap E_{n, \eta, \gamma, \kappa}(T_{A,n}, u_n)}
\\
&\le K \cdot n B_n\cdot
\frac{1}{n} \sum_{h = \gamma n/B_n}^{\gamma^{-1} n/B_n}
\frac{1}{\kappa^2 n}\cdot
\Pr{\left\{\left|\sum_{k=1}^{\eta B_n} \sum_{j=1}^k \frac{3 n \sigma_{k,j}^2}{2 B_n^2 h} A_{k,j}(v^\ast_h) - 1\right| > \varepsilon\right\}
\cap E_{n, \eta, \gamma, \kappa}(\mathsf{Cut}_{v^\ast_h}(T_\infty), v^\ast_h)}
\\
&\le \frac{K}{\gamma \kappa^2} 
\sup_{\gamma n/B_n \le h \le \gamma^{-1} n/B_n}
\Pr{\left\{\left|\sum_{k=1}^{\eta B_n} \sum_{j=1}^k \frac{3 n \sigma_{k,j}^2}{2 B_n^2 h} A_{k,j}(v^\ast_h) - 1\right| > \varepsilon\right\}
\cap E_{n, \eta, \gamma, \kappa}(\mathsf{Cut}_{v^\ast_h}(T_\infty), v^\ast_h)}.
\end{align*}
We finally treat the last probability; let us write
\[\sum_{k=1}^{\eta B_n} \sum_{j=1}^k \frac{3 n \sigma_{k,j}^2}{2 B_n^2 h} A_{k,j}(v^\ast_h)
= \sum_{i=1}^h \frac{3 n}{2 B_n^2 h} X_i,\]
where $X_i$ takes the value $\sigma_{k,j}^2$ if and only if the vertex at height $i-1$ on the spine has $k \le \eta B_n$ children, and the vertex at height $i$ on the spine is its $j$-th child, whereas $X_i$ takes the value $0$ whenever $k > \eta B_n$. Recall that on the spine, the vertices reproduce independently according to the size-biased law $\sum_{k \ge 1} k \mu_\q(k) \delta_k$, and furthermore, conditional on the number of children of its parent, the position of a vertex amongst its sibling is uniformly chosen.
According to~\eqref{eq:variance_accroissements_labels_cas_gaussien}, we thus have as $n \to \infty$
\begin{align*}
\Es{\sum_{k=1}^{\eta B_n} \sum_{j=1}^k \frac{3 n \sigma_{k,j}^2}{2 B_n^2 h} A_{k,j}(v^\ast_h)}
&= \frac{3 n}{2 B_n^2} \sum_{k=1}^{\eta B_n} \sum_{j=1}^k \sigma_{k,j}^2 \mu_\q(k)
\\
&= \frac{3 n}{2 B_n^2} \sum_{k=1}^{\eta B_n} \frac{k(k-1)}{3} \mu_\q(k)
\\
&\sim \frac{n}{2 B_n^2} \Var(\xi \ind{\xi \le \eta B_n}),
\end{align*}
where $\xi$ is distributed according to $\mu_\q$. Recall that the function $l(x) = \Var(\xi \ind{\xi \le x})$ is slowly varying so the factor $\eta$ in the last line above can be removed and we conclude from~\cite[Equation~7]{Kortchemski:Sub_exponential_tail_bounds_for_conditioned_stable_Bienayme_Galton_Watson_trees}, that our expectation tends to $1$ as $n \to \infty$. Replacing $\varepsilon$ by $2\varepsilon$ in all the preceding equations, we may thus replace the factor $1$ that we subtract in the probability that we are currently aiming at controlling by the preceding expectation. An application fo Markov inequality then shows that the probability
\[\Prc{\left\{\left|\sum_{k=1}^{\eta B_n} \sum_{j=1}^k \frac{3 n \sigma_{k,j}^2}{2 B_n^2 h} A_{k,j}(v^\ast_h) - \Es{\sum_{k=1}^{\eta B_n} \sum_{j=1}^k \frac{3 n \sigma_{k,j}^2}{2 B_n^2 h} A_{k,j}(v^\ast_h)}\right| > \varepsilon\right\}}
{E_{n, \eta, \gamma, \kappa}(\mathsf{Cut}_{v^\ast_h}(T_\infty), v^\ast_h)}\]
is bounded above by
\[\varepsilon^{-2} h \left(\frac{3 n}{2 B_n^2 h}\right)^2 \Es{X_1^2}.\]
Observe from~\eqref{eq:variance_accroissements_labels_cas_gaussien} that $\sigma^2_{k,j} \le k/2$ for every pair $1 \le j \le k$ so, as previously,
\[\Es{X_1^2} = \sum_{k=1}^{\eta B_n} \sum_{j=1}^k \sigma_{k,j}^4 \mu_\q(k)
\le \sum_{k=1}^{\eta B_n} \frac{k^3}{6} \mu_\q(k)
\le \frac{\eta B_n}{6} \sum_{k=1}^{\eta B_n} k^2 \mu_\q(k)
\sim \frac{\eta}{6} B_n l(B_n).\]
Whence for every $h \in [\gamma n/B_n, \gamma^{-1} n/B_n]$
\[h \left(\frac{3 n}{2 B_n^2 h}\right)^2 \Es{X_1^2}
\le \frac{9 \eta n^2 B_n l(B_n)}{24 B_n^4 h} (1+o(1))
\le \frac{\eta}{\gamma} \frac{9n l(B_n)}{24 B_n^2} (1+o(1)),\]
and the last fraction tends to $9/12$ as $n \to \infty$. So finally, we have for every $n$ large enough,
\[\Pr{\left\{\left|\sum_{k=1}^{\Delta(T_{A,n})} \sum_{j=1}^k \frac{3 n \sigma_{k,j}^2}{2 B_n^2 |u_n|} A_{k,j}(u_n) - 1\right| > \varepsilon\right\} \cap E_{n, \eta, \gamma, \kappa}(T_{A,n}, u_n)}
\le \frac{K \eta}{\gamma^2 \kappa^2}.\]
Since $\eta$ can be chosen arbitrarily small, the first convergence in~\eqref{eq:moyenne_accroissements_etiquettes_cas_gaussien} follows.

The second convergence in~\eqref{eq:moyenne_accroissements_etiquettes_cas_gaussien} follows by the exact same calculations: we can bound the probability
\[\Pr{\left\{\sum_{k=1}^{\Delta(\mathsf{Cut}_{u_n}(T_{A,n}))} \sum_{j=1}^k \left(\frac{n \sigma_{k,j}^2}{B_n^2 |u_n|} \right)^2 A_{k,j}(u_n) > \varepsilon\right\} \cap E_{n, \eta, \gamma, \kappa}(\mathsf{Cut}_{u_n}(T_{A,n}), u_n)}\]
by
\[\frac{1}{\varepsilon} \frac{B_n}{\gamma n} \frac{n^2}{B_n^4} \Var(\xi \ind{\xi \le \eta B_n})
\sim \frac{1}{\varepsilon\gamma} \frac{n}{B_n^3} l(B_n)
\sim \frac{2}{\varepsilon \gamma B_n}
\to 0,\]
as $n \to \infty$, which concludes the proof.
\end{proof}

As in the case $\alpha < 2$, we close this section by sketching the argument for the multi-dimensional marginals. One of the differences is that now the contribution of the branch-points vanishes.

\begin{lem}\label{lem:deplacement_max_autour_site}
We have the convergence in probability
\[B_{\zeta(T_{A,n})}^{-1/2} \max_{u \in T_n} \left|\max_{1 \le i \le k_u} \ell(ui) - \min_{1 \le i \le k_u} \ell(ui)\right| \cvproba 0.\]
\end{lem}

\begin{proof}
Note that again, we may, and shall, replace $B_{\zeta(T_{A,n})}$ by $B_n$. 
We follow the proof of~\cite[Proposition 2]{Marzouk:Scaling_limits_of_random_bipartite_planar_maps_with_a_prescribed_degree_sequence} which dealt with trees `with a prescribed degree sequence' in the finite-variance regime. Recall that a uniform random bridge $X_k$ in $\mathcal{B}_k^+$ has the same law as the first $k$ steps of a random walk with step distribution $\sum_{i \ge -1} 2^{-i-2} \delta_i$ conditioned on being at $0$ at time $k$. According to Lemma 6 in~\cite{Marzouk:Scaling_limits_of_random_bipartite_planar_maps_with_a_prescribed_degree_sequence}, there exists two constants $c, C > 0$ such that for every $k \ge 1$ and $x \ge 0$, we have
\[\Pr{\max_{1 \le i \le k} X_{k,i} - \min_{1 \le i \le k} X_{k,i} > x} \le C \e^{-cx^2/k}.\]
Let $\zeta_k(T_{A,n})$ denote the number of individuals in $T_{A,n}$ with $k$ offsprings and let $\Delta(T_{A,n})$ be the largest offspring in $T_{A,n}$. Fix $\varepsilon > 0$ and recall the bound $\ln(1-x) \ge -\frac{x}{1-x}$ for $x < 1$. We then have
\begin{align*}
&\Prc{\max_{u \in T_n} \left|\max_{1 \le i \le k_u} \ell(ui) - \min_{1 \le i \le k_u} \ell(ui)\right| \le \varepsilon B_n^{1/2}}{T_{A,n}}
\\
&= \prod_{k=1}^{\Delta(T_{A,n})} \Prc{\max_{1 \le i \le k} X_{k,i} - \min_{1 \le i \le k} X_{k,i} \le \varepsilon B_n^{1/2}}{T_{A,n}}^{\zeta_k(T_{A,n})}
\\
&\ge \exp\left(- C \sum_{k=1}^{\Delta(T_{A,n})} \zeta_k(T_{A,n}) \e^{-c\varepsilon^2B_n/k} (1+o(1))\right),
\end{align*}
and the claim reduces to showing the convergence in probability
\[\sum_{k=1}^{\Delta(T_{A,n})} \zeta_k(T_{A,n}) \e^{-c\varepsilon^2B_n/k}
\cvproba 0.\]
Since $x \mapsto x^2\e^{-x}$ is decreasing on $[2, \infty)$, we have on a set of high probability as $n \to \infty$,
\[\sum_{k=1}^{\Delta(T_{A,n})} \zeta_k(T_{A,n}) \e^{-c\varepsilon^2B_n/k}
\le \sum_{k=1}^{\Delta(T_{A,n})} \frac{k^2 \zeta_k(T_{A,n})}{B_n^2} \times \frac{B_n^2}{\Delta(T_{A,n})^2} \e^{-c\varepsilon^2 B_n/\Delta(T_{A,n})}.\]
Recall that $\Delta(T_{A,n})$ is the largest jump plus one of the {\L}ukasiewicz path $W_n$ so, according to~\eqref{eq:Duquesne_Kortchemski}, the ratio $\Delta(T_{A,n})/B_n$ tends to $0$ in probability and so
\[\frac{B_n^2}{\Delta(T_{A,n})^2} \e^{-c\varepsilon^2 B_n/\Delta(T_{A,n})} \cvproba 0.\]
It only remains to prove that the sequence $\sum_{k=1}^{\Delta(T_{A,n})} \frac{k^2 \zeta_k(T_{A,n})}{B_n^2}$ is bounded in probability in the sense that
\[\lim_{K \to \infty} \limsup_{n \to \infty} \Pr{\sum_{k=1}^{\Delta(T_{A,n})} \frac{k^2 \zeta_k(T_{A,n})}{B_n^2} > K} = 0.\]
Let us intersect the preceding event with $\{\Delta(T_{A,n}) \le B_n\}$ whose probability tends to one. Let us translate our probability in terms of the {\L}ukasiewicz path: we aim at showing
\[\lim_{K \to \infty} \limsup_{n \to \infty} \Pr{\left\{\sum_{i=1}^{\zeta(T_{A,n})} \frac{(W_n(i+1)-W_n(i))^2}{B_n^2} > K\right\} \cap \left\{\max_{1 \le i \le \zeta(T_{A,n})} W_n(i+1)-W_n(i) \le B_n\right\}} = 0.\]
We then use the same reasoning as in the proof of Lemma~\ref{lem:moments_marche_Luka}. For a path $S$ and $n \ge 1$, let us denote by $\varsigma_{A,n}$ the time such that the $n$-th jump of $S$ with values in the set $A-1$ is its $\varsigma_{A,n}$-th jump in total.  Note that our event is shift-invariant so we may replace the excursion $W_n$ by the bridge $S_n$ obtained by conditioning the random walk $S$ with step distribution $\sum_{k \ge -1} \mu_\q(k+1) \delta_k$ to be at $-1$ after its $\varsigma_{A,n}$-th jump. Then, by cutting the time interval in two and using a time-reversibility property for the second half due to Kortchemski~\cite[Proposition~6.8]{Kortchemski:Invariance_principles_for_Galton_Watson_trees_conditioned_on_the_number_of_leaves}, we may in fact only consider the first half of the bridge, i.e. up to time $\varsigma_{A,n/2}$. The latter is absolutely continuous with respect to the unconditioned random walk so, if $(\xi_i)_{i \ge 1}$ are i.i.d. copies of a random variable $\xi$ sampled from $\sum_{k \ge -1} \mu_\q(k+1) \delta_k$ and if now $\varsigma_{A,n}$ denote the least $i \ge 1$ such that $\#\{i \in \{1, \dots, \varsigma_{A,n}\} : \xi_i \in A-1\}=n$, then there exists $C > 0$ such that
\begin{multline*}
\Pr{\left\{\sum_{i=1}^{\zeta(T_{A,n})} \frac{(W_n(i+1)-W_n(i))^2}{B_n^2} > K\right\} \cap \left\{\max_{1 \le i \le \zeta(T_{A,n})} W_n(i+1)-W_n(i) \le B_n\right\}}
\\
\le C \cdot \Pr{\left\{\sum_{i=1}^{\varsigma_{A,n/2}} \frac{\xi_i^2}{B_n^2} > K\right\} \cap \left\{\max_{1 \le i \le \varsigma_{A,n/2}} \xi_i \le B_n\right\}}.
\end{multline*}
Since $\varsigma_{A,n/2}/n$ converges almost surely to $1/(2\mu_\q(A))$ (see e.g.~\cite[Lemma~6.2]{Kortchemski:Invariance_principles_for_Galton_Watson_trees_conditioned_on_the_number_of_leaves}), we may replace $\varsigma_{A,n/2}$ by $n/(2\mu_\q(A))$. 
The Markov inequality then yields
\[\Pr{\left\{\sum_{i=1}^{n/(2\mu_\q(A))} \frac{\xi_i^2}{B_n^2} > K\right\} \cap \left\{\max_{1 \le i \le n/(2\mu_\q(A))} \xi_i \le B_n\right\}}
\le \frac{n}{2K \mu_\q(A) B_n^2} \Es{\xi^2 \ind{\xi \le B_n}},\]
which converges to $1/(K \mu_\q(A))$ and our claim follows.
\end{proof}

\begin{proof}[Proof of Proposition~\ref{prop:marginal_etiquettes_cas_gaussien} in the case $k \ge 2$]
Let us only restrict ourselves to the case $k=2$ to ease the notation since the general case hides no extra difficulty. We sample two independent uniform random vertices of $T_n$, say, $u_n$ and $v_n$, and we let $w_n$ be their most recent ancestor, we denote by $\hat{u}_n$ and $\hat{v}_n$ the children of $w_n$ which are respectively an ancestor of $u_n$ and $v_n$ in order to decompose
\[\ell(u_n) = \ell(w_n) + (\ell(\hat{u}_n)-\ell(w_n)) + (\ell(u_n) - \ell(\hat{u}_n)),\]
and similarly for $v_n$. The important observation is that, conditional on $T_{A,n}$, $u_n$ and $v_n$, the random variables $\ell(w_n)$, $\ell(u_n) - \ell(\hat{u}_n)$ and $\ell(v_n) - \ell(\hat{v}_n)$ are independent.
According to~\eqref{eq:Duquesne_Kortchemski}, we have
\[\frac{B_{\zeta(T_{A,n})}}{\zeta(T_{A,n})} \left(|w_n|, |u_n| - |\hat{u}_n|, |v_n| - |\hat{v}_n|\right)
\cvloi
\left(\min_{r \in [U, V]} \Hexc_r, \Hexc_U - \min_{r \in [U, V]} \Hexc_r, \Hexc_V - \min_{r \in [U, V]}\Hexc_r\right),\]
where $U$ and $V$ are i.i.d uniform random variables on $[0,1]$ independent of $\Hexc$. We claim that, jointly with~\eqref{eq:Duquesne_Kortchemski}, we have
\begin{equation}\label{eq:cv_labels_trois_parties}
\left(\frac{3 \zeta(T_{A,n})}{2 B_{\zeta(T_{A,n})}^2}\right)^{1/2} \left(\frac{\ell(w_n)}{\sqrt{|w_n|}}, \frac{\ell(u_n) - \ell(\hat{u}_n)}{\sqrt{|u_n| - |\hat{u}_n|}}, \frac{\ell(v_n) - \ell(\hat{v}_n)}{\sqrt{|v_n| - |\hat{v}_n|}}\right)
\cvloi \left(G_1, G_2, G_3\right),
\end{equation}
where $G_1$, $G_2$, $G_3$ are i.i.d. standard Gaussian random variables. This actually follows from the arguments used in the proof of Proposition~\ref{prop:marginal_etiquettes_cas_gaussien} in the case $k = 1$ which show not only the convergence of $\ell(u_n)$, but also that if $a_n$ is an ancestor of $u_n$ such that the ratio $|a_n|/|u_n|$ converges in probability to some $a \in (0,1)$ as $n \to \infty$, then we have
\[\left(\frac{3 \zeta(T_{A,n})}{2 B_{\zeta(T_{A,n})}^2}\right)^{1/2} \left(\frac{\ell(a_n)}{\sqrt{|a_n|}}, \frac{\ell(u_n) - \ell(a_n)}{\sqrt{|u_n| - |\hat{u}_n|}}\right)
\cvloi \left(G_1, G_2\right).\]
Indeed, replacing $u_n$ by $a_n$ only replaces $h$ in the last part of the proof in the case $k = 1$ by $ah(1+o(1))$ which shows the convergence of the first marginal, and similarly for the second; the joint convergence holds since they are independent. Recall from Lemma~\ref{lem:deplacement_max_autour_site} that the maximal displacement at a branch-point is small, then the preceding convergence implies that of the first two components in~\eqref{eq:cv_labels_trois_parties}. The convergence of the last one also holds since the role of $u_n$ and $v_n$ is symmetric and so~\eqref{eq:cv_labels_trois_parties} holds by independence.

Since Lemma~\ref{lem:deplacement_max_autour_site} implies that
\[\left(\frac{\zeta(T_{A,n})}{B_{\zeta(T_{A,n})}^2}\right)^{1/2} \left(\ell(\hat{u}_n) - \ell(w_n), \ell(\hat{v}_n) - \ell(w_n)\right)
\cvproba \left(0,0\right),\]
We conclude from~\eqref{eq:cv_labels_trois_parties} that the pair
\[B_{\zeta(T_{A,n})}^{-1/2} (\ell(u_n), \ell(v_n))\]
converges in distribution as $n \to \infty$ towards
\[\frac{2}{3} \left(\sqrt{\min_{r \in [U, V]} \Hexc_r} \cdot G_1 + \sqrt{\Hexc_U - \min_{r \in [U, V]} \Hexc_r} \cdot G_2, \sqrt{\min_{r \in [U, V]} \Hexc_r} \cdot G_1 + \sqrt{\Hexc_V - \min_{r \in [U, V]} \Hexc_r} \cdot G_3\right)
\]
which is indeed distributed as $(\Lab_U, \Lab_V)$.
\end{proof}

\section{Scaling limits of maps}
\label{sec:limite_cartes}

The main goal of this section is to prove Theorem~\ref{thm:cv_cartes}, we also state and prove scaling limits on the profile of distances in Theorem~\ref{thm:profil} below. We first prove that large pointed and non-pointed maps are close, in order to focus on pointed maps. Relying of the description of such maps by labelled trees, we then state and prove Theorem~\ref{thm:profil}. Finally, we prove Theorem~\ref{thm:cv_cartes} in the last three subsections.
The proof of tightness in Theorem~\ref{thm:cv_cartes} follows from the functional convergence in Theorem~\ref{thm:cv_serpents_cartes} as in the pioneer work of Le Gall~\cite{Le_Gall:The_topological_structure_of_scaling_limits_of_large_planar_maps} who considered maps pointed at the origin of the root-edge, see also~\cite{Le_Gall:Uniqueness_and_universality_of_the_Brownian_map,Le_Gall-Miermont:Scaling_limits_of_random_planar_maps_with_large_faces} for maps pointed as here; all these references rely on a different labelled tree obtained by the Bouttier--Di Francesco--Guitter bijection~\cite{Bouttier-Di_Francesco-Guitter:Planar_maps_as_labeled_mobiles}. We recast their proof using the Janson--Stef{\'a}nsson bijection~\cite{Janson-Stefansson:Scaling_limits_of_random_planar_maps_with_a_unique_large_face}.

Throughout this section, we fix $S \in \{V, E, F\}$, and for every $n \ge 1$, we sample a pointed map $(M_n, \star)$ from $\P^{\q, \bullet}_{S=n}$. Recall that $\zeta(M_n)$ denotes the number of edges of $M_n$. Recall that we associate with $(M_n, \star)$ a labelled tree $(T_n, \ell)$ with the same amount of edges $\zeta(M_n)$, and, as discussed in Section~\ref{sec:BGW}, $T_n$ has the law of $T_{A,n}$ where $A = \Z_+$ if $S=E$, and $A = 0$ if $S=V$, and $A=\N$ if $S=F$.

\subsection{On the behaviour of leaves in a large Bienaymé--Galton--Watson tree}

Recall that the leaves of the tree are in one-to-one correspondence with the vertex of $M_n$ different from the distinguished one; we shall need the following two estimates.
First, recall the notation $\lambda(T_n)$ for the number of leaves of $T_n$. For every $0 \le j \le \zeta(T_n)$, let further $\Lambda(T_n, j)$ denote the number of leaves amongst the first $j$ vertices of $T_n$ in lexicographical order, and make $\Lambda$ a continuous function on $[0, \zeta(T_n)]$ after linear interpolation.

\begin{lem}\label{lem:repartition_feuilles}
We have the convergence in probability
\[\left(\frac{\Lambda(T_n, \zeta(T_n) t)}{\lambda(T_n)} ; t \in [0,1]\right) \cvproba (t ; t \in [0,1]).\]
\end{lem}

\begin{proof}
Kortchemski~\cite[Corollary~3.3]{Kortchemski:Invariance_principles_for_Galton_Watson_trees_conditioned_on_the_number_of_leaves} (again for trees conditioned by the number of leaves, but it extends to the general case) proved that when we restrict to a time interval $[\eta, 1]$ with $\eta > 0$, the probability that $t\mapsto \frac{\Lambda(T_n, \zeta(T_n) t)}{\lambda(T_n)}$ deviates from the identity decays sub-exponentially fast. We can then extend to the whole segment $[0,1]$ to get our result by `mirror symmetry'. It suffices to observe that the `mirror' {\L}ukasiewicz path visits more leaves in its last $k$ steps than the original {\L}ukasiewicz path in its first $k$ steps; indeed in order to visit a vertex in the original lexicographical order, one must first visits all its ancestors, whereas in the `mirror' order, some of them have been already visited (the root of the tree for example).
\end{proof}

The preceding result states that the leaves of the tree are homogeneously spread. Note that we could replace the leaves by the vertices with offspring in a given set $B \subset \Z_+$. The next result states that the inverse of the number of leaves, normalised to have expectation $1$, converges to $1$ in $L^1$.

\begin{lem}\label{lem:biais_sites_GW}
We have the convergence in probability
\[\lim_{n \to \infty} \Es{\left|\frac{1}{\lambda(T_n)} \frac{1}{\E[\frac{1}{\lambda(T_n)}]} - 1\right|} = 0.\]
\end{lem}

This convergence is~\cite[Lemma~8]{Marzouk:Scaling_limits_of_random_bipartite_planar_maps_with_a_prescribed_degree_sequence} in the finite-variance regime; the proof applies \emph{mutatis mutandis} in our case since the arguments used there, which are due to Kortchemski~\cite{Kortchemski:Invariance_principles_for_Galton_Watson_trees_conditioned_on_the_number_of_leaves}, hold in the more our general case. 
Following arguments from~\cite{Abraham:Rescaled_bipartite_planar_maps_converge_to_the_Brownian_map, Bettinelli-Jacob-Miermont:The_scaling_limit_of_uniform_random_plane_maps_via_the_Ambjorn_Budd_bijection, Bettinelli-Miermont:Compact_Brownian_surfaces_I_Brownian_disks}, it was then shown in~\cite[Proposition~12]{Marzouk:Scaling_limits_of_random_bipartite_planar_maps_with_a_prescribed_degree_sequence} that Lemma~\ref{lem:biais_sites_GW} yields the following comparison between pointed and non-pointed maps.

\begin{prop}\label{prop:biais_cartes_Boltzmann_pointees}
Let $\phi : \Mapp \to \Map : (M, \star) \mapsto M$ and let $\phi_* \P^{\q, \bullet}_{S=n}$ be the push-forward measure induced on $\Map$ by $\P^{\q, \bullet}_{S=n}$, then
\[\left\|\P^\q_{S=n} - \phi_* \P^{\q, \star}_{S=n}\right\|_{TV} \cv 0,\]
where $\|\cdot\|_{TV}$ refers to the total variation norm.
\end{prop}

Indeed, one can bound this total variation distance by the expectation in Lemma~\ref{lem:biais_sites_GW} with $\lambda(T_n)-1$ instead of $\lambda(T_n)$.
Observe that if $(M_n, \star)$ is sampled from $\P^{\q, \bullet}_{S=n}$, then, conditional on $M_n$, the vertex $\star$ is uniformly distributed in $M_n$.

\subsection{Radius and profile}
\label{sec:profil}

Although for $\alpha \in (1,2)$, we shall only obtain a convergence along subsequences of the metric spaces, because these subsequential limits are not characterised, still we do obtain some information about distances in large maps. Recall that we work with pointed maps $(M_n, \star)$ sampled from $\P^{\q, \bullet}_{S=n}$, but according to the preceding section, this pair is close to a non-pointed map sampled from $\P^{\q}_{S=n}$, in which we sample a vertex uniformly at random so the next result also holds in this context. Recall that $\zeta(M_n)$ denotes the number of edges of $M_n$, let us denote by $\upsilon(M_n)$ its number of vertices.

Let
\[R(M_n) = \max_{x \in M_n} \dgr(x, \star)\]
be the \emph{radius} of the map; define also a point measure on $\Z_+$, called the \emph{profile of distances}, by
\[\rho_{M_n}(k) = \#\{x \in M_n : \dgr(x, \star)=k\}.
\qquad k \in \Z_+.\]
Finally, let $\Delta(M_n)$ be the longest distance in $M_n$ between $\star$ and the two extremities of the root-edge (the other extremity is at distance $\Delta(M_n)-1$).

\begin{thm}
\label{thm:profil}
Let $\overline{\Lab} = \sup_{t \in [0,1]} \Lab_t$ and $\underline{\Lab} = \inf_{t \in [0,1]} \Lab_t$ and observe that $\overline{\Lab}$ and $-\underline{\Lab}$ have the same law by symmetry. Then the following convergences in distribution hold as $n \to \infty$:
\begin{enumerate}
\item $B_{\zeta(M_n)}^{-1/2} R(M_n) \to \overline{\Lab} - \underline{\Lab}$;
\item $B_{\zeta(M_n)}^{-1/2} \Delta(M_n) \to \overline{\Lab}$;
\item For every continuous and bounded function $\varphi$,
\[\frac{1}{\upsilon(M_n)} \sum_{k \ge 0} \varphi(B_{\zeta(M_n)}^{-1/2} k) \rho_{M_n}(k)
\cvloi
\int_0^1 \varphi(\Lab_t - \underline{\Lab}) \d t.\]
\end{enumerate}
\end{thm}

\begin{proof}
We rely on the bijection with the labelled tree $(T_n, \ell)$. Let us set $\underline{L}_n = \min_{1 \le i \le \lambda(T_n)} L_n(i)-1$; in this bijection, we have
\[R(M_n) = \max_{0 \le i \le \zeta(T_n)} L_n(i) - \underline{L}_n,\]
so the first convergence immediately follows from Theorem~\ref{thm:cv_serpents_cartes}. Similarly, the root-vertex of the tree is the farthest extremity of the root-edge of $M_n$ from $\star$, so
\[\Delta(M_n) = - \underline{L}_n,\]
and the second convergence is again an immediate consequence of Theorem~\ref{thm:cv_serpents_cartes}. We need a little more work for the third assertion. Our argument shall also serve later in Section~\ref{sec:tension_cartes} and~\ref{sec:carte_brownienne}.

Recall the notation $\lambda(T_n)$ for the number of leaves of $T_n$, which equals $\upsilon(M_n)-1$, and $\Lambda(T_n, j)$ for the number of leaves amongst the first $j$ vertices of $T_n$ in lexicographical order. For every $1 \le i \le \lambda(T_n)$, let $g(i) \in \{1, \dots, \zeta(T_n)\}$ be the index such that $u_{g(i)}$ is the $i$-th leaf of $T_n$. Since $j \mapsto \Lambda(T_n, j)$ is non-decreasing, Lemma~\ref{lem:repartition_feuilles} is equivalent to
\begin{equation}\label{eq:approximation_sites_aretes_carte}
\left(\frac{g(\lambda(T_n) t)}{\zeta(T_n)} ; t \in [0,1]\right) \cvproba (t ; t \in [0,1]),
\end{equation}
where as usual, we have linearly interpolated $g$ between integer values.

. Then observe that
\begin{align*}
\frac{1}{\upsilon(M_n)-1} \sum_{k \ge 0} \varphi(B_{\zeta(M_n)}^{-1/2} k) \rho_{M_n}(k)
&= \frac{1}{\lambda(T_n)} \varphi(0) + \frac{1}{\lambda(T_n)} \sum_{i = 1}^{\lambda(T_n)} \varphi\left(B_{\zeta(T_n)}^{-1/2} \left(L_n(g(k)) - \underline{L}_n\right)\right)
\\
&= \frac{1}{\lambda(T_n)} \varphi(0) + \int_0^1 \varphi\left(B_{\zeta(T_n)}^{-1/2} \left(L_n(g(\lceil \lambda(T_n)t\rceil)) - \underline{L}_n\right)\right) \d t,
\end{align*}
which converges in law to $\int_0^1 \varphi(\Lab_t - \underline{\Lab}) \d t$ according to~\eqref{eq:approximation_sites_aretes_carte} and Theorem~\ref{thm:cv_serpents_cartes}.
\end{proof}

\subsection{The Gromov--Hausdorff--Prokhorov topology}

Let us next briefly define this topology used in Theorem~\ref{thm:cv_cartes} in a way that is tailored for our purpose. Let $(X, d_x, m_x)$ and $(Y, d_Y, m_y)$ be two compact metric spaces equipped with a Borel probability measure. A \emph{correspondence} between these spaces is a subset $R \subset X \times Y$ such that for every $x \in X$, there exists $y \in Y$ such that $(x,y) \in R$ and vice-versa. The \emph{distortion} of $R$ is defined as
\[\mathrm{dis}(R) = \sup\left\{\left|d_X(x,x') - d_Y(y,y')\right| ; (x,y), (x', y') \in R\right\}.\]
Then we define the Gromov--Hausdorff--Prokhorov distance between these spaces as the infimum of all those $\varepsilon > 0$ such that there exists a coupling $\nu$ between $m_X$ and $m_Y$ and a compact correspondence $R$ between $X$ and $Y$ such that
\[\nu(R) \ge 1-\varepsilon \quad\text{and}\quad \mathrm{dis}(R) \le 2 \varepsilon.\]
This definition is not the usual one and is due to Miermont~\cite[Proposition~6]{Miermont:Tessellations_of_random_maps_of_arbitrary_genus}. We refer to Section~6 for more details on the Gromov--Hausdorff--Prokhorov distance. Let us only recall that it makes separable and complete the set of isometry classes of compact metric spaces equipped with a Borel probability measure.

If $(M_n\setminus\{\star\}, \dgr, \pgr)$ is the metric measured space given by the vertices of $M_n$ different from $\star$, their graph distance \emph{in $M_n$} and the uniform probability measure, then the Gromov--Hausdorff--Prokhorov distance between $(M_n, \dgr, \pgr)$, and $(M_n\setminus\{\star\}, \dgr, \pgr)$ is bounded by one so it suffices to prove that from every increasing sequence of integers, one can extract a subsequence along which the convergence in distribution
\begin{equation}\label{eq:convergence_cartes_pointees}
\left(M_n\setminus\{\star\}, B_{\zeta(M_n)}^{-1/2} \dgr, \pgr\right)
\cvloi
(\Bmap, \dBmap, \pBmap),
\end{equation}
holds for the Gromov--Hausdorff--Prokhorov topology.

\subsection{Tightness of distances}

Recall that the leaves of the labelled tree $(T_n, \ell)$ associated with $(M_n, \star)$ are in bijection with the vertices of $M_n$ different from $\star$. As for the internal vertices of $T_n$, they are each identified with their last child and so to each such internal vertex corresponds a leaf (the end of the right-most ancestral line starting from them) and therefore a vertex of $M_n\setminus\{\star\}$. Let $\varphi : T_n \to M_n\setminus\{\star\}$ be the map which associates with each vertex of $T_n$ its corresponding vertex of $M_n$. Let us list the vertices of $T_n$ as $u_0 < u_1 < \dots < u_{\zeta(M_n)}$ in lexicographical order and for every $i,j \in \{0, \dots, \zeta(M_n)\}$, we set
\[d_n(i,j) = \dgr(\varphi(u_i), \varphi(u_j)),\]
where $\dgr$ is the graph distance of $M_n$. We then extend $d_n$ to a continuous function on $[0,n]^2$ by `bilinear interpolation' on each square of the form $[i,i+1] \times [j,j+1]$ as in~\cite[Section~2.5]{Le_Gall:Uniqueness_and_universality_of_the_Brownian_map} or~\cite[Section~7]{Le_Gall-Miermont:Scaling_limits_of_random_planar_maps_with_large_faces}.

Define for every $t \in [0,1]$:
\[H_{(n)}(t)= \frac{B_{\zeta(M_n)}}{\zeta(M_n)} H_n(\zeta(M_n) t),
\qquad\text{and}\qquad
L_{(n)}(t) = B_{\zeta(M_n)}^{-1/2} L_n(\zeta(M_n) t),\]
and for every $s,t \in [0,1]$:
\begin{align*}
d_{(n)}(s, t) &= B_{\zeta(M_n)}^{-1/2} d_n(\zeta(M_n) s, \zeta(M_n) t),
\\
D_{L_{(n)}}(s, t) &= L_{(n)}(s) + L_{(n)}(t) - 2 \max\left\{\min_{r \in [s \wedge t, s \vee t]} L_{(n)}(r); \min_{r \in [0, s \wedge t] \cup [s \vee t, 1]} L_{(n)}(r)\right\}.
\end{align*}
Using the triangle inequality at a vertex where a geodesic from $\varphi(u_i)$ to $\star$ and a geodesic from $\varphi(u_j)$ to $\star$ in $M_n$ merge, Le Gall~\cite[Equation 4]{Le_Gall:Uniqueness_and_universality_of_the_Brownian_map} (see also~\cite[Lemma 3.1]{Le_Gall:The_topological_structure_of_scaling_limits_of_large_planar_maps} for a detailed proof) obtained the bound
\begin{equation}\label{eq:bornes_distances_carte}
d_{(n)}(s, t) \le D_{L_{(n)}}(s, t) + 2 B_{\zeta(M_n)}^{-1/2},
\end{equation}
for every $s,t \in [0,1]$ such that both $\zeta(M_n) s$ and $\zeta(M_n) t$ are integers, but then also in the other cases. Let us point out that this bound was obtained using the coding of the Bouttier--Di Francesco--Guitter bijection, where $L_n$ is the so-called \emph{white label function} of the two-type tree in the contour order. Nonetheless, as proved in~\cite[Lemma~1]{Marzouk:Scaling_limits_of_random_bipartite_planar_maps_with_a_prescribed_degree_sequence}, this process equals (deterministically) our process $L_n$ when the trees are related by the Janson--Stef{\'a}nsson bijection.

Recall from Section~\ref{sec:BGW} that $T_n$ has the law of $T_{A,n}$ where $A = \Z_+$ if $S=E$, and $A = 0$ if $S=V$, and $A=\N$ if $S=F$. Then Theorem~\ref{thm:cv_serpents_cartes} yields the convergence in distribution of continuous paths
\[\left(H_{(n)}(t), L_{(n)}(t), D_{L_{(n)}}(s, t)\right)_{s,t \in [0,1]}
\cvloi (\Hexc_t, \Lab_t, D_\Lab(s,t))_{s,t \in [0,1]},\]
where, similarly to the discrete setting,
\[D_\Lab(s,t) = \Lab_s + \Lab_t - 2 \max\left\{\min_{r \in [s \wedge t, s \vee t]} \Lab_r; \min_{r \in [0, s \wedge t] \cup [s \vee t, 1]} \Lab_r\right\}.\]
The bound~\eqref{eq:bornes_distances_carte} then easily shows that $d_{(n)}$ is tight. Therefore, from every increasing sequence of integers, we can extract a subsequence along which we have
\begin{equation}\label{eq:convergence_distances_sous_suite}
\left(H_{(n)}(t), L_{(n)}(t), d_{(n)}(s, t)\right)_{s,t \in [0,1]}
\cvloi
(\Hexc_t, \Lab_t, \dBmap(s,t))_{s,t \in [0,1]},
\end{equation}
where $(\dBmap(s,t))_{s,t \in [0,1]}$ depends a priori on the subsequence and satisfies $\dBmap \le D_\Lab$, see~\cite[Proposition~3.2]{Le_Gall:The_topological_structure_of_scaling_limits_of_large_planar_maps} for a detailed proof in a similar context.

In the next subsections, we implicitly restrict ourselves to a subsequence along which~\eqref{eq:convergence_distances_sous_suite} holds.

\subsection{Tightness of metric spaces}
\label{sec:tension_cartes}

Appealing to Skorokhod's representation Theorem, let us assume that the convergence~\eqref{eq:convergence_distances_sous_suite} holds almost surely (along the appropriate subsequence). We claim that, deterministically, the convergence~\eqref{eq:convergence_cartes_pointees} then holds. Let us first construct the limit space. As limit of the sequence $(d_{(n)})_{n \ge 1}$, the fonction $\dBmap$, which is continuous on $[0,1]^2$, is a pseudo-distance. We then define an equivalence relation on $[0,1]$ by setting
\[s \approx t
\qquad\text{if and only if}\qquad
\dBmap(s,t) = 0,\]
and we let $\Bmap$ be the quotient $[0,1] / \approx$, equipped with the metric induced by $\dBmap$, which we still denote by $\dBmap$. We let $\Pi$ be the canonical projection from $[0,1]$ to $\Bmap$ which is continuous (since $\dBmap$ is) so $(\Bmap, \dBmap)$ is a compact metric space, which finally we endow with the Borel probability measure $\pBmap$ given by the push-forward by $\Pi$ of the Lebesgue measure on $[0,1]$.

Recall our definition of the Gromov--Hausdorff--Prokhorov distance. Recall from Section~\ref{sec:profil} that for every $1 \le i \le \lambda(T_n)$, we denote by $g(i) \in \{1, \dots, \zeta(T_n)\}$ the index such that $u_{g(i)}$ is the $i$-th leaf of $T_n$, so the sequence $(\varphi(u_{g(i)}))_{1 \le i \le \lambda(T_n)}$ lists \emph{without redundancies} the vertices of $M_n$ different from $\star$. The set
\[\mathscr{R}_n = \left\{\left(\varphi(u_{g(\lceil \lambda(T_n) t\rceil)}), \Pi(t)\right) ; t \in [0,1]\right\}.\]
is a correspondence between $(M_n^\star\setminus\{\star\}, B_{\zeta(M_n)}^{-1/2} \dgr, \pgr)$ and $(\Bmap, \dBmap, \pBmap)$. Let further $\nu$ be the coupling between $\pgr$ and $\pBmap$ given by
\[\int_{M_n^\star\setminus\{\star\} \times \Bmap} \phi(v, x) \d\nu(v,x) = \int_0^1 \phi\left(\varphi(u_{g(\lceil \lambda(T_n) t\rceil)}), \Pi(t)\right) \d t,\]
for every test function $\phi$. Then $\nu$ is supported by $\mathscr{R}_n$ by construction. Finally, the distortion of $\mathscr{R}_n$ is given by
\[\sup_{s,t \in [0,1]} \left|d_{(n)}\left(\frac{g(\lceil \lambda(T_n) s\rceil)}{\zeta(T_n)}, \frac{g(\lceil \lambda(T_n) t\rceil)}{\zeta(T_n)}\right) - \dBmap(s,t)\right|,\]
which, appealing to~\eqref{eq:approximation_sites_aretes_carte}, tends to $0$ whenever the convergence~\eqref{eq:convergence_distances_sous_suite} holds, which concludes the proof of the tightness.

\subsection{Characterisation of the limit in the Brownian case}
\label{sec:carte_brownienne}

In this last subsection, we assume that $\alpha=2$ and we prove that~\eqref{eq:convergence_distances_sous_suite} holds without extracting a subsequence, and then so does~\eqref{eq:convergence_cartes_pointees}, with a limit which we next recall, following Le Gall~\cite{Le_Gall:The_topological_structure_of_scaling_limits_of_large_planar_maps} to which we refer for details.
First, we view $D_\Lab$ as a function on the tree $\CRT$ by setting
\[D_\Lab(x,y) = \inf\left\{D_\Lab(s,t) ; s,t \in [0,1], x=\pi(s) \text{ and }  y=\pi(t)\right\},\]
for every $x, y \in \CRT$, where we recall the notation $\pi : [0,1] \to \CRT = [0,1] / \sim$ for the canonical projection. Then we put
\[\dBmap^\ast(x,y) = \inf\left\{\sum_{i=1}^k D_\Lab(a_{i-1}, a_i) ; k \ge 1, (x=a_0, a_1, \dots, a_{k-1}, a_k=y) \in \CRT\right\}.\]
The function $\dBmap^\ast$ is a pseudo-distance on $\CRT$ which can be seen as a pseudo-distance on $[0,1]$ by setting $\dBmap^\ast(s,t) = \dBmap^\ast(\pi(s),\pi(t))$ for every $s,t \in [0,1]$.

As functions on $\CRT^2$, we clearly have $\dBmap^\ast \le D_\Lab$ and in fact, $\dBmap^\ast$ is the largest pseudo-distance on $\CRT$ satisfying this property. Indeed, if $D$ is another such pseudo-distance, then for every $x,y \in \CRT$, for every $k \ge 1$ and every $a_0, a_1, \dots, a_{k-1}, a_k \in \CRT$ with $a_0=x$ and $a_k=y$, by the triangle inequality, $D(x, y) \le \sum_{i=1}^k D(a_{i-1}, a_i) \le \sum_{i=1}^k D_\Lab(a_{i-1}, a_i)$ and so $D(x, y) \le \dBmap^\ast(x, y)$. Furthermore, if we view $\dBmap^\ast$ as a function on $[0,1]^2$, then for all $s,t \in [0,1]$ such that $\pi(s) = \pi(t)$ we have $\dBmap^\ast(\pi(s),\pi(t)) = 0$. We deduce from the previous maximality property that $\dBmap^\ast$ is the largest pseudo-distance $D$ on $[0,1]$ satisfying the following two properties:
\[D \le D_\Lab
\qquad\text{and}\qquad
\pi(s) = \pi(t) \quad\text{implies}\quad D(s,t) = 0.\]
We point out that since $\Hexc$ is $\sqrt{2}$ times the standard Brownian excursion, then our process $\Lab$ corresponds to $(\frac{8}{9})^{1/4} Z$ where $Z$ is used to define the Brownian map in~\cite{Le_Gall:The_topological_structure_of_scaling_limits_of_large_planar_maps} and in subsequent paper so the standard Brownian map is $(\Bmap, (\frac{9}{8})^{1/4}\dBmap^\ast, \pBmap)$.

Let $\dBmap$ be a limit in~\eqref{eq:convergence_distances_sous_suite} and note that it satisfies the two preceding properties; we claim that
\[\dBmap = \dBmap^\ast \qquad\text{almost surely}.\]
Our argument is adaped from~\cite[below Equation~27]{Marzouk:Scaling_limits_of_random_bipartite_planar_maps_with_a_prescribed_degree_sequence}, which itself was adapted from the work of Bettinelli \& Miermont~\cite[Lemma 32]{Bettinelli-Miermont:Compact_Brownian_surfaces_I_Brownian_disks}. 
According to the maximality property of $\dBmap^\ast$, the bound $\dBmap \le \dBmap^\ast$ holds almost surely so it suffices to prove that if $X, Y$ are i.i.d. uniform random variables on $[0,1]$ such that the pair $(X, Y)$ is independent of everything else, then
\begin{equation}\label{eq:identite_distances_carte_brownienne_repointee}
\dBmap(X,Y) \eqloi \dBmap^\ast(X,Y).
\end{equation}
It is known~\cite[Corollary 7.3]{Le_Gall:Uniqueness_and_universality_of_the_Brownian_map} that the right-hand side is distributed as $\dBmap^\ast(s_\star,Y) = \Lab_Y - \Lab_{s_\star}$, where $s_\star$ is the (a.s. unique~\cite{Le_Gall-Weill:Conditioned_Brownian_trees}) point at which $\Lab$ attains its minimum. The point is that, in the discrete setting, $d_n$ describes the distances in the map between the vertices $(\varphi(u_i))_{0 \le i \le \zeta(M_n)}$, and some vertices of $M_n$ may appear more often that others in this sequence so if one samples two uniform random times, they do not correspond to two uniform random vertices of the map. Nonetheless, this effect disappears at the limit, according to~\eqref{eq:approximation_sites_aretes_carte}. Indeed, fix $X, Y$ i.i.d. uniform random variables on $[0,1]$ such that the pair $(X, Y)$ is independent of everything else, and set $x = \varphi(u_{g(\lceil \lambda(T_n) X\rceil)})$ and $y = \varphi(u_{g(\lceil \lambda(T_n) Y\rceil)})$. Note that $x$ and $y$ are uniform random vertices of $M_n \setminus \{\star\}$, they can therefore be coupled with two independent uniform random vertices $x'$ and $y'$ of $M_n$ in such a way that the conditional probability given $M_n$ that $(x,y) \ne (x', y')$ is at most $2/\upsilon(M_n) \to 0$ as $n \to \infty$; we implicitly assume in the sequel that $(x,y) = (x', y')$. Since $\star$ is also a uniform random vertex of $M_n$, we obtain that
\begin{equation}\label{eq:identite_distances_carte_discrete_repointee}
\dgr(x,y) \eqloi \dgr(\star, y).
\end{equation}
By definition,
\[\dgr(x,y) = d_n(g(\lceil \lambda(T_n) X\rceil), g(\lceil \lambda(T_n) Y\rceil)).\]
Recall that the labels on $T_n$ describe the distances from $\star$ in $M_n$, we therefore have
\[\dgr(\star, y) = L_n(g(\lceil \lambda(T_n) Y\rceil)) - \min_{0 \le j \le \zeta(T_n)} L_n(j) + 1.\]
We obtain~\eqref{eq:identite_distances_carte_brownienne_repointee} by letting $n \to \infty$ in~\eqref{eq:identite_distances_carte_discrete_repointee} along the same subsequence as in~\eqref{eq:convergence_distances_sous_suite}, appealing also to~\eqref{eq:approximation_sites_aretes_carte}.


\end{document}